 \documentclass[10pt]{article}
\usepackage[top=1in, bottom=1in, left=1in, right=1in]{geometry}

\usepackage{amsmath}
\usepackage{amssymb,amsfonts,amsmath,amsthm,amscd,dsfont,mathrsfs,mathtools,microtype,nicefrac,pifont,natbib}

\usepackage[linktocpage,colorlinks,linkcolor=blue,anchorcolor=blue,citecolor=blue,urlcolor=blue,pagebackref]{hyperref}

\renewcommand*{\backref}[1]{\ifx#1\relax \else Page #1 \fi}
\renewcommand*{\backrefalt}[4]{%
  \ifcase #1 \footnotesize{(Not cited.)}%
  \or        \footnotesize{(Cited on page~#2.)}%
  \else      \footnotesize{(Cited on pages~#2.)}%
  \fi
}

\usepackage{cleveref,enumitem}

\usepackage{xcolor,amsthm}
\usepackage{natbib}
\setcitestyle{authoryear,open={(},close={)}} 

\usepackage{mathtools}

\newcommand{\lv}{\left\lVert}
\newcommand{\rv}{\right\rVert}

\newcommand{\mb}{\mathbb}
\newcommand{\mc}{\mathcal}

\newcommand{\cwpi}{C_{\text{WPI}}}
\newcommand{\TX}{\Tilde{X}}
\newcommand{\fl}[1]{\lfloor #1 \rfloor}

\newtheorem{theorem}{Theorem}
\newtheorem{lemma}{Lemma}
\newtheorem{remark}{Remark}
\newtheorem{corollary}{Corollary}
\newtheorem{definition}{Definition}
\newtheorem{proposition}{Proposition}
\newtheorem{assump}{Assumption}

\newenvironment{myassump}[2][]
  {\renewcommand\theassump{#2}\begin{assump}[#1]}
  {\end{assump}}
\usepackage{mathtools}

\usepackage{lastpage}

\allowdisplaybreaks

\hypersetup{
  colorlinks,
  linkcolor={blue},
  citecolor={blue},
  urlcolor={blue}
}

\author{Ye He\thanks{Department of Mathematics, University of California, Davis. \texttt{leohe@ucdavis.edu} }
\and Tyler Farghly\thanks{Department of Statistics, University of Oxford, UK. \texttt{farghly@stats.ox.ac.uk}}
\and Krishnakumar Balasubramanian\thanks{Department of Statistics, University of California, Davis. \texttt{kbala@ucdavis.edu}} 
\and Murat A. Erdogdu\thanks{Department of Computer Science and Department of Statistical Sciences, University of Toronto. \texttt{erdogdu@cs.toronto.edu} }
}

\begin{document}
\title{Mean-Square Analysis of Discretized It\^o Diffusions \\
for Heavy-tailed Sampling}

\maketitle

\begin{abstract}
We analyze the complexity of sampling from a class of heavy-tailed distributions by discretizing a natural class of It\^o diffusions associated with weighted Poincar\'e inequalities. Based on a mean-square analysis, we establish the iteration complexity for obtaining a sample whose  distribution is $\epsilon$ close to the target distribution in the Wasserstein-2 metric. In this paper, our results take the mean-square analysis to its limits, i.e., we invariably only require that the target density has finite variance, the minimal requirement for a mean-square analysis. To obtain explicit estimates, we compute upper bounds on certain moments associated with heavy-tailed targets under various assumptions. We also provide similar iteration complexity results for the case where only function evaluations of the unnormalized target density are available by estimating the gradients using a Gaussian smoothing technique. We provide illustrative examples based on the multivariate $t$-distribution.

\end{abstract}

\section{Introduction}
The problem of sampling from a given target density $\pi: \mathbb{R}^d \to \mathbb{R}$ arises in a wide variety of problems in statistics, machine learning, operations research and applied mathematics. Markov chain Monte Carlo (MCMC) algorithms are a popular class of algorithms for sampling~\citep{robert1999monte,andrieu2003introduction, hairer2006geometric,brooks2011handbook,meyn2012markov, leimkuhler2016molecular,douc2018markov}; a widely used approach in this domain is to discretize an It\^o diffusion that has the target as its stationary density. A popular choice of diffusion is the overdamped Langevin diffusion,
\begin{align}\label{LDy}
    dX_t=\nabla \log \pi(X_t)  dt +\sqrt{2} d B_t,
\end{align}
where $B_t$ is a $d$-dimensional Brownian motion. For example, the Unadjusted Langevin Algorithm~\citep{rossky1978brownian}, the Metropolis Adjusted Langevin Algorithm~\citep{roberts1996exponential,roberts1998optimal} and the proximal sampler~\citep{titsias2018auxiliary,lee2021structured,vono2022efficient} arise as different discretizations of~\eqref{LDy}. Under light-tailed assumptions, i.e. when the density $\pi$ has exponentially fast decaying tails, the diffusion $X_t$ in~\eqref{LDy} converges exponentially fast to $\pi$ as its stationary density, which motivates the use of discretizations of~\eqref{LDy} as practical algorithms for sampling. In the last decade, the non-asymptotic iteration complexity of various discretizations have been well-explored, thereby providing a relatively comprehensive story of sampling from light-tailed densities. 

Motivated by applications in robust statistics~\citep{kotz2004multivariate, jarner2007convergence, kamatani2018efficient}, multiple comparison procedures~\citep{genz2004approximations, genz2009computation}, Bayesian statistics~\citep{gelman2008weakly,ghosh2018use}, and statistical machine learning~\citep{balcan2017sample, nguyen2019non, simsekli2020fractional,diakonikolas2020learning}, in this work, we are interested in sampling from densities that have heavy-tails, for example, those with tails that are polynomially decaying. When the target density $\pi$ is heavy-tailed, the solution to~\eqref{LDy} does not converge exponentially to its stationary density in various metrics of interest. Indeed, Theorem 2.4 by~\cite{roberts1996exponential} shows that if $|\nabla \log \pi (x)| \to 0$ as $|x| \to \infty$, then the solution to~\eqref{LDy} is \emph{not} exponentially ergodic. In the other direction, standard results in the literature, for example~\cite{wang2006functional, bakry2014analysis} show that the solution to~\eqref{LDy} converging exponentially fast to its equilibrium density in the $\chi^2$ metric, is equivalent to the density $\pi$ satisfying the Poincar\'e inequality, which in turn requires $\pi$ to have exponentially decaying tails. Furthermore, when $\pi$ has  polynomially decaying tails, the convergence is only sub-exponential or polynomial~\cite[Chapter 4]{wang2006functional}. Consequently, the algorithms obtained as discretizations of the Langevin diffusion in~\eqref{LDy} are suited to sampling only from light-tailed exponentially decaying densities, and are rather inefficient for sampling from heavy-tailed densities. 

Our approach to heavy-tailed sampling is hence based on discretizing certain natural It\^o diffusions that arise in the context of the following Weighted Poincar\'e inequality~\citep{blanchet2009asymptotics, bobkov2009weighted}. Such inequalities could be considered generalizations of the  Brascamp-Lieb inequality (established for the class of log-concave densities) to a class of heavy-tailed densities. 

\begin{theorem}[Weighted Poincar\'e Inequality; Theorem 2.3 in \cite{bobkov2009weighted}]\label{thm:2.3} 
Let the target density be of the form $\pi_\beta\propto V^{-\beta}$ with $\beta> d$ and $V\in \mc{C}^2(\mb{R}^d)$ positive, convex and with $(\nabla^2 V)^{-1}(x)$ well-defined for all $x\in \mb{R}^d$. For any smooth and $\pi_\beta$-integrable function $g$ on $\mb{R}^d$ and $G=Vg$,
\begin{equation}\label{eq:thm 2.3}
    (\beta+1)Var_{\pi_\beta}(g)\le \int_{\mb{R}^d} \frac{\langle (\nabla^2 V)^{-1}\nabla G,
    \nabla G \rangle}{V} d\pi_\beta +\frac{d}{\beta-d}\left( \int_{\mb{R}^d} g d\pi_\beta \right).
\end{equation}
\end{theorem}
\noindent A canonical example of a heavy-tailed density that satisfies the conditions in Theorem~\ref{thm:2.3}, and hence~\eqref{eq:thm 2.3}, is the multivariate $t$-distribution. In particular, we consider the following It\^o diffusion process 

\begin{equation}\label{eq:diffusion strongly convex V}
    dX_t=-(\beta-1)\nabla V(X_t) dt+\sqrt{2V(X_t)} dB_t,
\end{equation}
where $(B_t)_{t\ge 0}$ is a standard Brownian motion in $\mb{R}^d$. The It\^o diffusion in~\eqref{eq:diffusion strongly convex V} converges exponentially fast to the target $\pi_\beta$ in the $\chi^2$-divergence as long as it satisfies the Weighted Poincar\'e inequality and additional mild assumptions; see Proposition~\ref{prop:chi square cts strongly convex V} for details. Hence, we study the oracle complexity of the Euler-Maruyama discretization of~\eqref{eq:diffusion strongly convex V}, for sampling from heavy-tailed densities. Our proofs are based on \emph{mean-square analysis} techniques, a popular technique to analyze numerical discretizations of stochastic differential equations; see, for example,~\cite{milstein2004stochastic} for an overview. Our results in this paper pushes mean-square analysis to its limits; the heavy-tailed densities we consider invariably need to have only finite variance, which is the minimum requirement when using this technique. 

\subsection{Our Contributions}
In this work, we make the following contributions:
\begin{itemize}[noitemsep]
    \item In Theorem~\ref{thm:strongly convex V W2 decay}, we provide upper bounds on the number of iterations required by the Euler-Maruyama discretization of~\eqref{eq:diffusion strongly convex V} to obtain a sample that is $\epsilon$-close in the Wasserstein-2 metric to the target density. The established bounds are in terms of certain (first and second-order) moments of the target density $\pi$.  Our proof technique is based on a mean-squared analysis; we demonstrate that for the case of multivariate $t$-distributions, our analysis is non-vacuous as long as the density has finite variance, a necessary condition to carry out the mean-squared analysis. 
    \item While the result in Theorem~\ref{thm:strongly convex V W2 decay} assumes access to the exact gradient of the unnormalized target density function (referred to as the first-order setting), in Theorem~\ref{thm:zeroth order algorithm strongly convex V}, we analyze the case when the gradient is estimated based on function evaluations (the zeroth-order setting) based on a Gaussian smoothing technique. 
    \item We provide several illustrative examples highlighting the differences between the results in the first and the zeroth-order setting. Specifically, in Section~\ref{subsec:exponential contractivity} we show that for the multivariate $t$-distribution with smaller degrees of freedom, (and hence the truly heavy-tailed case) the gradient estimation error is dominated by the discretization error. Whereas, in the case with larger degrees of freedom (and hence the comparatively moderately heavy-tailed case), the discretization error is of comparable order to the gradient estimation error. Hence, the zeroth-order algorithm matches the iteration complexity of the first-order algorithm by using mini-batch gradient estimators. 
\end{itemize}

\subsection{Related Work}
Non-asymptotic iteration complexity of different discretizations of~\eqref{LDy} have been analyzed extensively in the last decade. The analysis of the Unadjusted Langevin Algorithm (ULA) under various light-tailed assumptions was carried out, for example, in \cite{dalalyan2017theoretical, durmus2017nonasymptotic, dalalyan2019user,durmus2019analysis, lee2020logsmooth, shen2019randomized, he2020ergodicity, chen2020fast, durmus2019analysis, dalalyan2019bounding, li2020riemannian, chen2020fast, chewi2021analysis} and references therein. In particular,~\cite{vempala2019rapid,erdogdu2021convergence, chewi2021analysis} analyzed the performance of ULA under various functional inequalities suited to light-tailed densities. Furthermore, the recent work of~\cite{balasubramanian2022towards} analyzed the performance of (averaged) ULA for target densities that are only H\"{o}lder continuous, albeit in the weaker Fisher information metric.

Several works, for example,~\cite{dwivedi2019log, chewi2021optimal, wu2022minimax}, analyzed the Metropolis-Adjusted Langevin Algorithm (MALA) in light-tailed settings. The proximal sampler algorithm was analyzed under various light-tailed assumptions in~\cite{lee2021structured, chen2022improved}. 
The iteration complexity of the widely used Hamiltonian Monte Carlo algorithm and discretizations of underdamped Langevin diffusions were analyzed, for example,  in~\cite{dalalyan2020sampling,bou2020coupling,chen2020fast, ma2021there,monmarche2021high,cao2021complexity, wang2022accelerating,chen2022optimal}. We also refer interested readers to~\cite{lu2022complexity,ding2021langevin} for non-asymptotic analyses of other MCMC algorithms used in practice in light-tailed settings.

In the context of heavy-tailed sampling,~\cite{kamatani2018efficient} considered the scaling limits of appropriately modified Metropolis random walk in an asymptotic setting. \cite{johnson2012variable} proposed a variable transformation method in the context of Metropolis Random Walk algorithms. Here, the heavy-tailed density is converted into a light-tailed one based on certain invertible transformations so that one can leverage the rich literature on light-tailed sampling algorithms. Similar ideas were also examined recently in~\cite{yang2022stereographic}. It is also worth highlighting that~\cite{deligiannidis2019exponential,durmus2020geometric} and~\cite{bierkens2019ergodicity} used the transformation approach for proving asymptotic exponential ergodicity of bouncy particle and zig-zag samplers respectively, in the heavy-tailed setting. We also point out the recent works of~\cite{andrieu2021subgeometric} and \cite{andrieu2021comparison} that establish similar sub-exponential ergodicity results for other sampling methods such as the piecewise deterministic Markov process Monte Carlo, independent Metropolis-Hastings sampler and pseudo-marginal methods in the polynomially heavy-tailed setting. The works of~\cite{simsekli2020fractional, huang2021approximation} and~\cite{zhang2022ergodicity} established exponential ergodicity results for diffusions driven by $\alpha$-stable processes with heavy-tailed densities as its equilibrium in the continuous-time setting. However, the problem of obtaining convergence results for practical discretizations of these diffusions is still largely open.

The literature on non-asymptotic oracle complexity analysis of heavy-tailed sampling is extremely limited.~\cite{chandrasekaran2009sampling} considered the iteration complexity of Metropolis random walk algorithm for sampling from $s$-concave distributions.~\cite{he2022heavy} considered ULA on a class of transformed densities (i.e., the heavy-tailed density is transformed to a light-tailed one with an invertible transformation, similar to~\cite{johnson2012variable}) and established non-asymptotic oracle complexity results. However, they focused mainly on the case of isotropic densities.~\cite{li2019stochastic} analyzed a class of discretizations of general It\^o diffusions that admit heavy-tailed equilibrium densities. A detailed comparison to~\cite{li2019stochastic} is provided in Section~\ref{subsec:exponential contractivity}.

The recent works by~\cite{hsieh2018mirrored,zhang2020wasserstein, chewi2020exponential,ahn2021efficient,jiang2021mirror,li2022mirror} also considered sampling based on discretizations of the Mirror Langevin diffusions. The above-mentioned  works mainly focus on sampling from constrained densities. The continuous-time convergence is analyzed typically under the so-called mirror Poincar\'e inequalities which are generalizations of the Brascamp-Lieb inequalities in a different direction compared to the Weighted Poincar\'e inequalities. The discretization analysis by~\cite{li2022mirror} is based on mean-squared analysis.

As mentioned previously, our work leverages the literature on weighted functional inequalities, that are satisfied by heavy-tailed densities. The weighted Poincare inequality was introduced in~\cite{blanchet2009asymptotics} and \cite{bobkov2009weighted}, and using an extension of the Brascamp-Lieb inequality, is shown to hold for the class of $s$-concave densities. We also refer the interested reader to~\cite{cattiaux2010functional,cattiaux2011some,bonnefont2016spectral,cordero2017transport, cattiaux2019entropic} for various extensions and improvements of the works of~\cite{blanchet2009asymptotics} and~\cite{bobkov2009weighted}.

\subsection{Notation}
We use the following notation throughout the rest of the paper.
\begin{itemize}[noitemsep]
    \item $\langle \cdot, \cdot \rangle$ denotes the Euclidean inner product and $|\cdot|$ denotes the Euclidean norm.
    \item For two matrices $A$ and $B$, $A \preceq B$ means that $B-A$ is positive semi-definite. The 2-norm of any $d\times d$ matrix $A$ is denoted as $\lv A \rv_2$. $I_d$ is the \(d \times d\) identity matrix. 
    \item $\Delta$ denotes the Laplacian, and $\nabla$ denotes the gradient of a given function.
    \item $\mc{C}^2(\mb{R}^d)$ refers to the set of all real functions on $\mb{R}^d$ that are twice continuously differentiable. $\mc{C}^2_c(\mb{R}^d)$ refers to the set of all functions in $\mc{C}^2(\mb{R}^d)$ with compact support.
    \item The Wasserstein-2 distance between two probability measures on $\mb{R}^d$, $\mu$ and $\nu$ is given by
    \begin{align*}
        W_2(\mu,\nu):=\inf_{\zeta\in C(\mu,\nu) } \left(\int_{\mb{R}^d \times \mb{R}^d} |x-y|^2 \zeta(dx,dy) \right)^{\frac{1}{2}}.
    \end{align*}
    where $C(\mu,\nu)$ is the set of all measures on $\mb{R}^d\times\mb{R}^d$ whose marginals are $\mu$ and $\nu$ respectively. 
    \item The $\chi^2$ divergence from a probability measure $\nu$ to a probability measure $\mu$ is defined as
    \begin{align*}
        \chi^2(\nu|\mu):=\int_{\mb{R}^d} \left( \frac{\nu(dx)}{\mu(dx)}-1\right)^2 \mu(dx).
    \end{align*}
    \item The gamma and beta functions are given by:
    \begin{align*}
    \Gamma(z):=\int_0^\infty t^{z-1} e^{-t} dt,~~\forall \ z>0, \quad\text{and}\quad       B(x,y):=\int_0^1 t^{x-1}(1-t)^{y-1} dt,~~\forall\ x,y>0.
    \end{align*}
    \item For two positive quantities $f(d),g(d)$ depending on $d$, we define $f(d)= O(g(d))$ if there exists a constant $C>0$ such that $f(d)\le C g(d)$ for all $d>1$. We define $f(d)=\Theta(g(d))$ if there exist constants $C_1,C_2>0$ such that $C_1 g(d)\le f(d)\le C_2 g(d)$ for all $d>1$. We use $\tilde{O}$ to hide $\log$ factors in the $O$ notation.\end{itemize}

\subsection{Organization}

In Section \ref{sec:weighted Poincare inequality}, we first establish the exponential ergodicity of the It\^o diffusion in~\eqref{eq:diffusion strongly convex V} under certain assumptions that are favorable for the discretization analysis. We next provide our main results on the non-asymptotic oracle complexity of the Euler-Maruyama discretization of~\eqref{eq:diffusion strongly convex V}. In Section~\ref{subsec:calculations on expectations}, we provide moment computations in the heavy-tailed setting that are required to obtain explicit rates from the results in Section~\ref{sec:weighted Poincare inequality}. In Section~\ref{sec:zosetting}, we provide an extension of our results to the zeroth-order setting. In Section~\ref{subsec:exponential contractivity} we provide several illustrative examples. We discuss further implications of our assumptions in Section~\ref{sec:further}. The proofs are provided in Section~\ref{sec:proof} and in Appendices~\ref{sec:appgencase},~\ref{cauchyexample} and~\ref{sec:cauchy application}.

\section{It\^o Discretizations and Weighted Poincare inequalities}\label{sec:weighted Poincare inequality}

In this section, our goal is to analyze the It\^o diffusion in~\eqref{eq:diffusion strongly convex V} which admits a specific class of heavy-tailed densities as its stationary density. Let $X_0$ follow distribution $\rho_0$ and denote the distribution of $X_t$ by $\rho_t$ for all $t\ge 0$. For any function $\psi\in \mc{C}^2_c(\mb{R}^d)$, the infinitesimal generator of \eqref{eq:diffusion strongly convex V} is given by
\begin{align}\label{eq: generator convex V}
    \mc{L}\psi=-(\beta-1)  \langle \nabla V , \nabla \psi \rangle + V \Delta \psi.
\end{align}
Hence, the Fokker-Planck equation corresponding to \eqref{eq:diffusion strongly convex V} is
\begin{align}\label{eq:FPE strongly convex V}
    \partial_t \rho_t=\nabla\cdot \left( \beta\rho_t \nabla V+V\nabla \rho_t \right) =\nabla \cdot \left( \rho_t V \nabla \log \frac{\rho_t}{\pi_\beta} \right).
\end{align}
It follows that, under the conditions in Theorem~\ref{thm:2.3}, $\pi_\beta \propto V^{-\beta}$ is the unique stationary density of \eqref{eq:diffusion strongly convex V}. We next examine the convergence properties of~\eqref{eq:diffusion strongly convex V} to its stationary density. To do so, we introduce the following assumption. 

\begin{assump}\label{ass:V1} 
There exists a positive constant $C_V$ such that, for all $x\in\mb{R}^d$, 
\begin{align*}
\frac{\langle (\nabla^2 V)^{-1}(x)\nabla V(x),\nabla V(x) \rangle}{V(x)}\le C_V. 
\end{align*}
\end{assump}

When $V$ is radially symmetric, i.e., when \(V(x) := \phi(|x|)\) for some \(\phi \in \mc{C}^2(\mb{R_+})\), the condition in Assumption~\ref{ass:V1} simplifies as follows. Note that 
\begin{align*}
\nabla V(x) = \frac{\phi'(|x|) }{|x|} x , \qquad \text{and}\qquad \nabla^2 V = \left(\phi''(|x|) - \frac{\phi'(|x|)}{|x|}\right) \frac{x \otimes x}{|x|^2} + \frac{\phi'(|x|)}{|x|} I_d,
\end{align*}
where $\otimes$ denotes outer-product. Hence, it follows that it is sufficient for \(\phi\) to satisfy $$\phi'(r) \leq (\phi''(r) r) \wedge (C_V \phi(r)/r), \text{ for all } r \geq 0.$$ For example, this property holds with \(C_V = p\) if \(\phi\) is a \(p\)-order polynomial with \(p \geq 2\) and non-negative coefficients. 

We next provide the following corollary to Theorem \ref{thm:2.3}, motivated by the discussion in Section 2 of \cite{bobkov2009weighted}.

\begin{corollary}\label{lem:WPI convex V} Consider the setting of Theorem \ref{thm:2.3} and suppose further that Assumption \ref{ass:V1} holds with $C_V\in (0,\beta+1)$, then for any smooth, $\pi_\beta$-integrable function, $\phi$ on $\mb{R}^d$,
\begin{align}\label{eq:WPI convex V}
    Var_{\pi_\beta}(\phi) &\le \left( \sqrt{\beta+1}-\sqrt{C_V} \right)^{-2} \int_{\mb{R}^d} \langle V(x)(\nabla^2 V)^{-1}(x)\nabla \phi(x),\nabla \phi(x) \rangle \pi_\beta(x)dx.
\end{align}
\end{corollary}
\begin{proof}[Proof]
We start from \eqref{eq:thm 2.3}, assume that $\int_{\mb{R}^d} g d\pi_\beta=0$. Then \eqref{eq:thm 2.3} could be rewritten as
\begin{align*}
    (\beta+1)\int_{\mb{R}^d} g(x)^2 \pi_\beta(x)dx\le \int_{\mb{R}^d} \frac{\langle (\nabla^2 V)^{-1}(x) \nabla (gV)(x),\nabla(gV)(x) \rangle}{V(x)} \pi_\beta(x)dx.
\end{align*}
Now, note that we have the following elementary bound 
\begin{align*}
    \langle A(u+v),(u+v) \rangle\le r\langle Au,u \rangle+\frac{r}{r-1} \langle Av,v \rangle,\quad u,v\in \mb{R}^d, r>1,
\end{align*}
for any arbitrary positive definite symmetric matrix $A\in \mb{R}^{d\times d}$. Hence, we obtain
\begin{align*}
    (\beta+1)\int_{\mb{R}^d} g(x)^2 \pi_\beta(x)dx & \le r \int_{\mb{R}^d} \frac{\langle (\nabla^2 V)^{-1}(x)g(x)\nabla V(x),g(x)\nabla V(x) \rangle}{V(x)}\pi_\beta(x)dx\\
    &\ +\frac{r}{r-1}\int_{\mb{R}^d}\frac{\langle (\nabla^2 V)^{-1}(x) V(x)\nabla g(x),V(x)\nabla g(x) \rangle}{V(x)}\pi_\beta(x)dx.
\end{align*}
Invoking the condition in Assumption \ref{ass:V1}, we further obtain
\begin{align*}
    (\beta+1)\int_{\mb{R}^d} g(x)^2 \pi_\beta(x)dx & \le rC_V \int_{\mb{R}^d} g(x)^2 \pi_\beta(x)dx \\ 
    &\qquad \qquad +\frac{r}{r-1}\int_{\mb{R}^d} \langle V(x)(\nabla^2 V)^{-1}(x)\nabla g(x),\nabla g(x) \rangle \pi_\beta(x)dx, 
\end{align*}
which then implies that, for any \(r \in (1, (\beta + 1)/C_V)\),
\begin{align*}
\int_{\mb{R}^d} g(x)^2 \pi_\beta(x)dx &\le \frac{r}{(r-1)(\beta+1-r C_V)} \int_{\mb{R}^d} \langle V(x)(\nabla^2 V)^{-1}(x)\nabla g(x),\nabla g(x) \rangle \pi_\beta(x)dx.
\end{align*}
With the choice of $r\coloneqq \sqrt{\frac{\beta+1}{C_V}}>1$, we get that for all $g$ such that $\int g d\pi_\beta=0$, and 
\begin{align*}
    \int_{\mb{R}^d} g(x)^2 \pi_\beta(x)dx &\le \left( \sqrt{\beta+1}-\sqrt{C_V} \right)^{-2} \int_{\mb{R}^d} \langle V(x)(\nabla^2 V)^{-1}(x)\nabla g(x),\nabla g(x) \rangle \pi_\beta(x)dx.
\end{align*}
For all general $\phi$, letting $g=\phi-\int \phi d\pi_\beta$, we get
\begin{align*}
    Var_{\pi_\beta}(\phi) &\le \left( \sqrt{\beta+1}-\sqrt{C_V} \right)^{-2} \int_{\mb{R}^d} \langle V(x)(\nabla^2 V)^{-1}(x)\nabla \phi(x),\nabla \phi(x) \rangle \pi_\beta(x)dx.
\end{align*}
\end{proof}
When $V$ is strongly convex, Assumption \ref{ass:V1} holds under the following sufficient condition.
\begin{assump}\label{ass:WPI strongly convex V} The function $V:\mb{R}^d\to (0,\infty)$ is twice continuously differentiable and V satisfies
\begin{itemize}
    \item [(1)] $V$ is $\alpha$-strongly convex, i.e. $\nabla^2 V(x)\succeq \alpha I_d $ for all $x\in \mb{R}^d$.
    \item [(2)] There exists a positive constant $C_V$ such that, for all $x\in\mb{R}^d$,
    $$
    \frac{\langle \nabla V(x),\nabla V(x) \rangle}{V(x)}\le \alpha C_V.
    $$
\end{itemize}
\end{assump}

The following result follows immediately from Assumption~\ref{ass:WPI strongly convex V}.
\begin{lemma}\label{cor:WPI strongly convex V} Let $\beta>d$. If Assumption \ref{ass:WPI strongly convex V} holds with $C_V\in (0,\beta+1)$, then for any smooth, $\pi_\beta$ integrable function $\phi$ on $\mb{R}^d$, we have \begin{align}\label{eq:WPI strongly convex V}
    \text{Var}_{\pi_\beta}(\phi) &\le \alpha^{-1}\left( \sqrt{\beta+1}-\sqrt{C_V} \right)^{-2} \int_{\mb{R}^d} V(x)|\nabla \phi(x)|^2 \pi_\beta(x)dx.
\end{align}
\end{lemma}

 With \eqref{eq:WPI strongly convex V}, we can show the exponential decay in $\chi^2$-divergence along \eqref{eq:diffusion strongly convex V}. The proof of the following proposition is standard and we include it here for completeness. 
\begin{proposition}\label{prop:chi square cts strongly convex V} Under the conditions in Lemma \ref{cor:WPI strongly convex V}, for $(X_t)$ following diffusion \eqref{eq:diffusion strongly convex V} with $\rho_t$ being the distribution of $X_t$, we have
\begin{align}\label{eq:chi square cts strongly convex}
    \chi^2(\rho_t|\pi_\beta)\le \exp\left( -2\alpha\left(\sqrt{\beta+1}-\sqrt{C_V}\right)^2 t \right)\chi^2(\rho_0|\pi_\beta).
\end{align}
\end{proposition}
\begin{proof}[Proof of Proposition~\ref{cor:WPI strongly convex V}]\label{pf:chi square cts strongly convex}
First we can calculate the derivative of $\chi^2(\rho_t|\pi)$ via \eqref{eq:FPE strongly convex V},
\begin{align*}
    \frac{d}{dt} \chi^2(\rho_t|\pi_\beta)&= \frac{d}{dt}\int_{\mb{R}^d} \left(\frac{\rho_t(x)}{\pi_\beta(x)}-1\right)^2 \pi_\beta(x)dx \\
    &=2\int_{\mb{R}^d} \partial_t \rho_t(x) \left(\frac{\rho_t(x)}{\pi_\beta(x)}-1\right) dx \\
    &=-2\int_{\mb{R}^d} \left\langle \nabla\left(\frac{\rho_t}{\pi_\beta}\right)(x),\nabla \log \left(\frac{\rho_t}{\pi_\beta}\right)(x)\right\rangle V(x)\rho_t(x)dx \\
    &=-2\int_{\mb{R}^d} V(x)\left|\nabla \left(\frac{\rho_t}{\pi_\beta}\right)(x)\right|^2\pi_\beta(x)dx.
\end{align*}
According to \eqref{eq:WPI strongly convex V}, we get
\begin{align*}
    \frac{d}{dt} \chi^2(\rho_t|\pi_\beta)&\le -2\alpha\left( \sqrt{\beta+1}-\sqrt{C_V} \right)^{2} \text{Var}_{\pi_\beta}\left(\frac{\rho_t}{\pi_\beta}\right) \\
    &=-2\alpha\left( \sqrt{\beta+1}-\sqrt{C_V} \right)^{2}\chi^2(\rho_t|\pi_\beta).
\end{align*}
Finally, \eqref{eq:chi square cts strongly convex} follows from Gronwall's inequality.
\end{proof}

The above result shows that for the class of $\pi_\beta$ satisfying Assumption~\ref{ass:WPI strongly convex V}, the It\^o diffusion in~\eqref{eq:diffusion strongly convex V}, converges exponentially fast to its stationary density. Hence, time-discretizations of~\eqref{eq:diffusion strongly convex V} provide a practical way of sampling from that class of densities. The Euler-Maruyama discretization to \eqref{eq:diffusion strongly convex V} is given by
\begin{align}\label{eq:EM strongly convex V}
    x_{k+1}=x_k-h(\beta-1)\nabla V(x_k)+\sqrt{2hV(x_k)}\xi_{k+1},
\end{align}
where $h>0$ is the step size and $\{\xi\}_{k=1}^\infty$ is a sequence of i.i.d. standard Gaussian random vectors in $\mb{R}^d$. We now present our main result on the iteration complexity of~\eqref{eq:EM strongly convex V} for sampling from $\pi_\beta$. We state our discretization result, based on a mean-square analysis, in the $W_2$ metric. In particular, we highlight that Proposition~\ref{prop:chi square cts strongly convex V} requires that condition that $\beta > d$, in addition to Assumption \ref{ass:WPI strongly convex V}, whereas Theorem~\ref{thm:strongly convex V W2 decay} below, does not. In Section~\ref{sec:further}, we revisit these conditions and provide additional insights. Obtaining convergence results in the stronger $\chi^2$-divergence is left as future work.

\begin{theorem}\label{thm:strongly convex V W2 decay} Let $V$ be gradient-Lipschitz with parameter $L>0$, and satisfying Assumption \ref{ass:WPI strongly convex V} with 
\begin{align}\label{eq:delta}
\delta\coloneqq \frac{\beta-1-\frac{1}{4}C_V d}{\frac{1}{4}C_V d}>0.
\end{align}
Let $(x_k)_{k=0}^\infty$ be generated from \eqref{eq:EM strongly convex V} with $\nu_k$ denoting the distribution of $x_k$, for all $k\ge 0$. Then with the step-size,
\begin{align*}
h<\min\left(\frac{1}{4(\beta-1)L},\frac{2\delta}{3(1+\delta)\alpha(\beta-1)}\right),
\end{align*}
the decay of Wasserstein-2 distance along the Markov chain $(x_k)_{k=0}^\infty$ can be described by the following equation: For all $k\ge 1$,
\begin{align}\label{eq:strongly convex V W2 convergence}
    W_2(\nu_k,\pi_\beta)\le (1-A)^k W_2(\nu_0,\pi_\beta)+\frac{C}{A}+\frac{B}{\sqrt{A(2-A)}}.
\end{align}
with $A,B$ and $C$ given respectively in \eqref{eq:strongly convex V parameters A}, \eqref{eq:strongly convex V parameters B} and \eqref{eq:strongly convex V parameters C}.
\end{theorem}

\begin{remark}[Constant $\delta$]
We now motivate the definition and the condition on the constant $\delta$ based on exponential contractivity arguments.
\begin{definition}[Exponential contractivity]\label{def:exp contraction} Let $X_t$, $Y_t$ be two different solutions to the same stochastic differential equation (SDE) with initial conditions $x,y$ respectively. We say the SDE is $W_2$-exponential contractive if there exists a constant $\kappa>0$, such that 
\begin{align*}
    W_2(L(X_t),L(Y_t))\le e^{-\kappa t}~|x-y|, 
\end{align*}
where by $L(X)$ we refer to the law of $X$.
\end{definition}
Uniform dissipativity is a sufficient condition for exponential contractivity \cite[Theorem 10]{gorham2019measuring}. The uniform dissipativity condition for \eqref{eq:diffusion strongly convex V} can be represented as
\begin{align*}
     -(\beta-1)\langle \nabla V(x)-\nabla V(y),x-y \rangle+\frac{1}{2} \lv \sqrt{2V(x)}I_d-\sqrt{2V(y)} I_d \rv_F^2\le -\kappa |x-y|^2,
\end{align*}
or equivalently as
\begin{align*}
    -(\beta-1)\langle \nabla V(x)-\nabla V(y),x-y \rangle+ d |\sqrt{V(x)}-\sqrt{V(y)}|^2 \le -\kappa |x-y|^2.
\end{align*}
When $V$ satisfies Assumption \ref{ass:WPI strongly convex V}, a sufficient condition for the above uniform dissipativity condition is given by 
\begin{align*}
    & -\alpha (\beta-1) |x-y|^2+ \frac{d}{4}\alpha C_V |x-y|^2\le -\kappa |x-y|^2,
\end{align*}
or equivalently, 
\begin{align*}
 \alpha\left( \beta-1-\frac{d}{4}C_V  \right)\le \kappa.
\end{align*}
The sufficient condition coincides with the condition that $\delta>0$ in Theorem \ref{thm:strongly convex V W2 decay}, which also motivates the assumption in Theorem \ref{thm:strongly convex V W2 decay}.

\end{remark}

\begin{remark}[Iteration complexity] \label{rem:iter}
With Theorem \ref{thm:strongly convex V W2 decay}, we can calculate the order of the iteration complexity to reach an $\epsilon$-accuracy in Wasserstein-2 distance. With \eqref{eq:strongly convex V parameters A},\eqref{eq:strongly convex V parameters B},\eqref{eq:strongly convex V parameters C}, we have 
\begin{align*}
    \frac{C}{A}&=\frac{9(\delta+1)L}{\alpha \delta}d^{\frac{1}{2}} h^{\frac{1}{2}}\mb{E}_{\pi_\beta}\left[ V(X) \right]^{\frac{1}{2}}+\frac{6(\delta+1)L}{\alpha \delta} (\beta-1)h \mb{E}_{\pi_\beta}\left[ |\nabla V(X)|^2 \right]^{\frac{1}{2}}, \\
    \frac{B}{\sqrt{A(2-A)}}&\le \frac{8(\delta+3)}{\delta}d^{\frac{1}{2}}h^{\frac{1}{2}}\mb{E}_{\pi_\beta}\left[ V(X) \right]^{\frac{1}{2}}+\frac{8(\delta+3)}{\delta} (\beta-1)h  \mb{E}_{\pi_\beta}\left[ |\nabla V(X)|^2 \right]^{\frac{1}{2}}. \\
\end{align*}
The above display implies that
\begin{align*}
    \frac{C}{A}+\frac{B}{\sqrt{A(2-A)}}&\le \frac{9(\delta+3)}{\delta}\left(1+\frac{L}{\alpha}\right)  \left( d^{\frac{1}{2}}h^{\frac{1}{2}}\mb{E}_{\pi_\beta}\left[ V(X) \right]^{\frac{1}{2}}+ (\beta-1)h  \mb{E}_{\pi_\beta}\left[ |\nabla V(X)|^2 \right]^{\frac{1}{2}} \right).
\end{align*}
Hence, we get $\frac{C}{A}+\frac{B}{\sqrt{A(2-A)}}<\epsilon/2$ if the step-size $h$ satisfies
\begin{align}\label{eq:strongly convex V step size}
    h< \min \left\{\frac{\delta^2 \mb{E}_{\pi_\beta}\left[ V(X) \right]^{-1}\epsilon^2 }{81 d(\delta+3)^2(1+\frac{L}{\alpha})^2}, \frac{\delta \mb{E}_{\pi_\beta}\left[ |\nabla V(X)|^2 \right]^{-\frac{1}{2}} \epsilon }{81 (\beta-1) (\delta+3) (1+\frac{L}{\alpha})} \right\}.
\end{align}
Defining $K_\epsilon=\log \left({2W_2(\nu_0,\pi_\beta)}/{\epsilon}\right)$, we have $W_2(\nu_k,\pi_\beta)<\epsilon$ for all $k\ge K$ with 
\begin{align}\label{eq:strongly convex V iteration complexity K}
    K&= \frac{3(1+\delta)}{\alpha(\beta-1)\delta h^*} K_\epsilon \nonumber \\
    &\le 273 \max\left\{ \frac{(\delta+3)^3 (1+\frac{L}{\alpha})^2 d \mb{E}_{\pi_\beta}\left[ V(X) \right]}{\alpha \delta^3 (\beta-1) \epsilon^2 }, \frac{(\delta+3)^2 (1+\frac{L}{\alpha}) \mb{E}_{\pi_\beta}\left[ |\nabla V(X)|^2 \right]^{\frac{1}{2}}  }{\alpha \delta^2 \epsilon} \right\} K_\epsilon.
\end{align}
Recall the definition of $\delta$ in~\eqref{eq:delta}. The order of $K$ depends on the order of $\delta$. That is, we have the following two cases:
\begin{itemize}
    \item If $\delta = O(1)$ and $\beta = O(d)$, we have that 
    $$
       K= \Tilde{O}\left(\frac{1}{\alpha \epsilon^2 } \left(1+\frac{L}{\alpha}\right)^2\mb{E}_{\pi_\beta}\left[ V(X) \right]+ \frac{1}{\alpha \epsilon}\left(1+\frac{L}{\alpha}\right)\mb{E}_{\pi_\beta}\left[ |\nabla V(X)|^2 \right]^{\frac{1}{2}}  \right). 
     $$
    \item  If $\delta = O(1/d)$ and $\beta = O(d)$, we have that $$
     K = \Tilde{O}\left( \frac{d^3}{\alpha \epsilon^2}\left(1+\frac{L}{\alpha}\right)^2  \mb{E}_{\pi_\beta}\left[ V(X) \right]+ \frac{d^2}{\alpha\epsilon}\left(1+\frac{L}{\alpha}\right) \mb{E}_{\pi_\beta}\left[ |\nabla V(X)|^2 \right]^{\frac{1}{2}} \right).
    $$
\end{itemize}
\end{remark}

In order to obtain more explicit iteration complexity bounds from  Remark~\ref{rem:iter}, it is required to compute bounds on the following two quantities: $\mb{E}_{\pi_\beta}\left[|\nabla V(X)|^2\right]$ and $\mb{E}_{\pi_\beta}\left[ V(X) \right]$.
\section{Moment Bounds}\label{subsec:calculations on expectations}
 In this section, we compute moment bounds under the conditions in Theorem \ref{thm:strongly convex V W2 decay}. 

\subsection{An Example: Multivariate $t$-distribution}
We first start with the isotropic case.  
\begin{proposition}\label{prop:expectations under Cauchy} Let $\pi_\beta=Z_{\beta}^{-1} V^{-\beta}$ with $\beta>d/2+1$, $V(x)=1+|x|^2$ and $Z_{\beta}=\int_{\mb{R}^d} (1+|x|^2)^{-\beta} dx $. We have
\begin{align}
    \mb{E}_{\pi_\beta}\left[ V(X) \right]=\frac{\beta-1}{\beta-1-\frac{d}{2}} \quad \text{and}\quad \mb{E}_{\pi_\beta}\left[ |\nabla V(X)|^2 \right]=\frac{2d}{\beta-1-\frac{d}{2}}. \label{eq:expectation under Cauchy gradient V}
\end{align}
\end{proposition}
\begin{proof}\label{pf:expectation under Cauchy} Let $A_{d}(1)$ denote the surface area of the unit sphere in $d$ dimensions. By a standard calculation, we have that, for all $\beta>\frac{d}{2}$,
\begin{align*}
    Z_\beta&= \int_{\mb{R}^d} (1+|x|^2)^{-\beta} dx=\int_0^\infty (1+r^2)^{-\beta} r^{d-1} dr A_{d}(1)=\frac{\pi^{\frac{d}{2}}}{\Gamma(\frac{d}{2})} \int_0^\infty (1+R)^{-\beta} R^{\frac{d}{2}-1} dR \\
    &=\frac{\pi^{\frac{d}{2}}}{\Gamma(\frac{d}{2})}  \int_0^1 u^{\frac{d}{2}-1}(1-u)^{\beta-\frac{d}{2}-1} du =\frac{\pi^{\frac{d}{2}}B(\frac{d}{2},\beta-\frac{d}{2})}{\Gamma(\frac{d}{2})},
\end{align*}
where $B$ is the beta function. In the above calculation, the second identity follows from a change to polar coordinates. The third identity follows from a substitution with $R=r^2$ and the fourth identity follows from a substitution $u={R}/{(1+R)}$.  Therefore for all $\beta>d/2+1$, we have that
\begin{align*}
    \mb{E}_{\pi_\beta}\left[ V(X) \right]&= Z_{\beta}^{-1} \int_{\mb{R}^d} (1+|x|^2) (1+|x|^2)^{-\beta} dx=\frac{Z_{\beta-1}}{Z_{\beta}} =\frac{\pi^{\frac{d}{2}}B(\frac{d}{2},\beta-1-\frac{d}{2})}{\Gamma(\frac{d}{2})}  \frac{\Gamma(\frac{d}{2})}{\pi^{\frac{d}{2}}B(\frac{d}{2},\beta-\frac{d}{2})}\\
    &=\frac{B(\frac{d}{2},\beta-1-\frac{d}{2})}{B(\frac{d}{2},\beta-\frac{d}{2})}=\frac{\Gamma(\frac{d}{2})\Gamma(\beta-1-\frac{d}{2})}{\Gamma(\beta-1)} \frac{\Gamma(\beta)}{\Gamma(\frac{d}{2})\Gamma(\beta-\frac{d}{2})}=\frac{\beta-1}{\beta-1-\frac{d}{2}}.
\end{align*}
where the fourth identity follows from the property of Beta function, $B(x,y)=\frac{\Gamma(x)\Gamma(y)}{\Gamma(x+y)}$ and the fifth identity follows from the property of $\Gamma$ function, $\Gamma(1+z)=z \Gamma(z)$. For the other expectation, we have
\begin{align*}
    \mb{E}_{\pi_\beta} \left[ |\nabla V(X)|^2 \right]&=Z_\beta^{-1}\int_{\mb{R}^d} |2x|^2 (1+|x|^2)^{-\beta} dx =4Z_{\beta}^{-1} A_{d-1}(1) \int_0^\infty r^2 (1+r^2)^{-\beta} r^{d-1} dr\\
    &=\frac{4 \pi^\frac{d}{2} }{\Gamma(\frac{d}{2})Z_\beta} \int_0^\infty R^\frac{d}{2}(1+R)^{-\beta} dR =\frac{4 \pi^\frac{d}{2} }{\Gamma(\frac{d}{2})Z_\beta} \int_0^1 u^{\frac{d}{2}}(1-u)^{\beta-\frac{d}{2}-2} du\\
    &=\frac{4 \pi^\frac{d}{2} B(\frac{d}{2}+1,\beta-\frac{d}{2}-1) }{\Gamma(\frac{d}{2})}  \frac{\Gamma(\frac{d}{2})}{\pi^{\frac{d}{2}}B(\frac{d}{2},\beta-\frac{d}{2})} =\frac{4B(\frac{d}{2}+1,\beta-\frac{d}{2}-1)}{B(\frac{d}{2},\beta-\frac{d}{2})}\\
    &=4\frac{\Gamma(\frac{d}{2}+1)\Gamma(\beta-\frac{d}{2}-1)}{\Gamma(\beta)}\frac{\Gamma(\beta)}{\Gamma(\frac{d}{2})\Gamma(\beta-\frac{d}{2})} =\frac{2d}{\beta-\frac{d}{2}-1},
\end{align*}
where we apply the same substitutions and properties of Beta functions and Gamma functions in the above calculation.
\end{proof}
\begin{remark}\label{rem:order estimation of expectations} If $\pi_\beta$ is the class of isotropic multivariate $t$-distributions, with the results in Proposition \ref{prop:expectations under Cauchy}, the order of the two expectations in terms of the dimension parameter $d$ is given as follows,
\begin{itemize}
    \item when $\beta>\frac{d}{2}+1$ and $\beta-1-\frac{d}{2}= O(d)$, we have
\begin{align*}
    \mb{E}_{\pi_\beta}\left[ V(X) \right] = O(1),\qquad\text{and}\qquad \mb{E}_{\pi_\beta}\left[ |\nabla V(X)|^2 \right] = O(1).
\end{align*}
    \item when $\beta>\frac{d}{2}+1$ and $\beta-1-\frac{d}{2}= O(1)$, we have
\begin{align*}
    \mb{E}_{\pi_\beta}\left[ V(X) \right] = O(d),\qquad\text{and}\qquad \mb{E}_{\pi_\beta}\left[ |\nabla V(X)|^2 \right] = O(d).
\end{align*}
\end{itemize}
\end{remark}
For a general class of non-isotropic multivariate t-distribution, we consider $\pi_\beta=Z_{\beta}^{-1}V^{-\beta}$ with $V(x)=1+x^T \Sigma x$ where $\Sigma$ is a strictly positive-definite $d\times d$ matrix. In \cite{roth2012multivariate}, it's been shown that for any $\beta>\frac{d}{2}$, the normalization constant is  $$Z_\beta=\frac{\Gamma(\frac{\nu}{2})\pi^{\frac{d}{2}}\sqrt{\text{det}(\Sigma)}}{\Gamma(\frac{\nu+d}{2})}=\frac{\Gamma(\beta-\frac{d}{2})\pi^{\frac{d}{2}}\sqrt{\text{det}(\Sigma)}}{\Gamma(\beta)}.$$
Therefore for any $\beta>\frac{d}{2}+1$, we have
\begin{align*}
    \mb{E}_{\pi_\beta}\left[V(X)\right]&=\frac{Z_{\beta-1}}{Z_{\beta}}=\frac{\Gamma(\beta)\Gamma(\beta-1-\frac{d}{2})}{\Gamma(\beta-1)\Gamma(\beta-\frac{d}{2})}=\frac{\beta-1}{\beta-1-\frac{d}{2}},
\end{align*}
and
\begin{align*}
    \mb{E}_{\pi_\beta} \left[ |\nabla V(X)|^2 \right] &=Z_\beta^{-1} \int_{\mb{R}^d} \langle \nabla V(x),V(x)^{-\beta} \nabla V(x) \rangle dx\\
    &=-Z_\beta^{-1} \int_{\mb{R}^d} V(x)\nabla \cdot \left( V(x)^{-\beta} \nabla V(x) \right) dx \\
    &=\beta \mb{E}_{\pi_\beta} \left[ |\nabla V(X)|^2 \right]-  Z_\beta^{-1} \int_{\mb{R}^d} \Delta V(x) V(x)^{-(\beta-1)} dx.
\end{align*}
The above identity implies
\begin{align*}
    \mb{E}_{\pi_\beta} \left[ |\nabla V(X)|^2 \right] &=(\beta-1)^{-1} Z_{\beta}^{-1} \int_{\mb{R}^d} \Delta V(x) V(x)^{-(\beta-1)}dx\\
    &\le (\beta-1)^{-1} Z_{\beta}^{-1} \int_{\mb{R}^d} \text{trace}(\nabla^2 V(x)) V(x)^{-(\beta-1)}dx \\
    &\le \frac{\text{trace}(\Sigma) }{ \beta-1}  \mb{E}_{\pi_\beta}\left[V(X)\right] \\
    &\le \frac{\text{trace}(\Sigma)}{\beta-1-\frac{d}{2}}  
\end{align*}
where the second inequality follows from the fact that $\nabla^2 V(x)=\Sigma$.
\begin{remark}\label{rem:order estimation of expectations nonisotropic} If $\pi_\beta$ is in the class of non-isotropic multivariate t-distributions, the order of the two expectations in terms of the dimension parameter $d$ is as follows, \begin{itemize}
    \item when $\beta>\frac{d}{2}+1$ and $\beta-1-\frac{d}{2}= O(d)$, we have
\begin{align*}
    \mb{E}_{\pi_\beta}\left[ V(X) \right] = O(1),\qquad\text{and}\qquad \mb{E}_{\pi_\beta}\left[ |\nabla V(X)|^2 \right] = O(d^{-1}\emph{trace}(\Sigma)).
\end{align*}
    \item when $\beta>\frac{d}{2}+1$ and $\beta-1-\frac{d}{2}= O(1)$, we have
\begin{align*}
    \mb{E}_{\pi_\beta}\left[ V(X) \right] = O(d),\qquad\text{and}\qquad \mb{E}_{\pi_\beta}\left[ |\nabla V(X)|^2 \right] = O(\emph{trace}(\Sigma)).
\end{align*}
\end{itemize}
\end{remark}
\subsection{Non-isotropic densities with quadratic-like $V$ outside of a ball}
In this section, we estimate the expectations for a class of non-isotropic densities in the form of $\pi_\beta\propto V^{-\beta}$ with $V$ satisfying the following Lyapunov condition: 
\begin{align}\label{eq:lypcondition}
\exists~\varepsilon,R>0~\text{such that}~\Delta V(x)-(\beta-1)\frac{|\nabla V(x)|^2}{V(x)}\le -\varepsilon~\qquad \forall~|x|\ge R.
\end{align}
The above Lyapunov condition characterizes the class of $V$ that are `quadratic-like' outside a ball of radius $R$. If we assume that $V$ has Lipschitz gradients, then when $\beta$ is sufficiently large, the above assumption is satisfied if $V$ satisfies the PL inequality $|\nabla V(x)|^2 \geq a^2 V(x)$ wherever $|x| \geq R$ with some $a > 0$ and it is from this inequality that quadratic growth follows. In particular, if $V$ satisfies the gradient Lipschitz assumption with parameter $L$, we have that for all $\beta\ge 1+a^{-2}(dL+\varepsilon)$,
\begin{align*}
    \Delta V(x)-(\beta-1)\frac{|\nabla V(x)|^2}{V(x)}\le dL-(\beta-1)a^2\le  -\varepsilon \qquad \forall\ |x|\ge R,
\end{align*}
thereby leading to the Lyapunov condition in~\eqref{eq:lypcondition}. 

\begin{proposition}\label{prop:expectation for V quadratic outside a ball}
If $V\in \mc{C}^2(\mb{R}^d)$ is positive, $L$-gradient Lipschitz and satisfies~\eqref{eq:lypcondition}, then we have
\begin{align}\label{eq:expectation for V quadratic outside a ball}
\mb{E}_{\pi_\beta}\left[V(x)\right]\le \left(dL+\varepsilon\right) \max_{|x|\le R} V(x),\quad\text{and}\quad\mb{E}_{\pi_\beta}\left[|\nabla V(X)|^2\right]\le \frac{d L\left(dL+\varepsilon\right)}{(\beta-1)}\max_{|x|\le R} V(X) .
\end{align}
\end{proposition}
\begin{proof}
Since \(\mc{L}\) is ergodic with stationary distribution $\pi_\beta$, we have $$\mb{E}_{\pi_\beta}\left[V(X)\right]=\lim_{t\to\infty} \mb{E}\left[V(X_t)\right],$$ with $(X_t)_{t\ge 0}$ being the solution to \eqref{eq:diffusion strongly convex V} with initial condition $X_0=x$. We will first bound $\mb{E}\left[V(X_t)\right]$ and then take $t\to\infty$. Let $(P_t)_{t\ge 0}$ be the Markov semigroup of $\eqref{eq:diffusion strongly convex V}$, then
  \begin{align*}
      \frac{d}{dt}\mb{E}_{\pi_\beta}\left[V(X_t)\right]&=\frac{d}{dt} P_t V(x)=P_t \mc{L}V(x).
  \end{align*}
  With \eqref{eq: generator convex V}, we have
  \begin{align*}
      \mc{L}V(x)&=V(x)\left[ \Delta V(x)-(\beta-1)\frac{|\nabla V(x)|^2}{V(x)} \right] \\
      &\le V(x)\left( -\varepsilon 1_{|x|\ge R}+dL 1_{|x|<R}  \right) \\
      &\le -\varepsilon V(x)+\left(dL+\varepsilon\right)\max_{|x|\le R} V(x),
  \end{align*}
where the first inequality follows from \eqref{eq:lypcondition} and the fact that $\Delta V\le d\lv \nabla^2 V \rv_2$. Therefore we obtain
  \begin{align*}
      \frac{d}{dt}P_tV(x)\le -\varepsilon P_t V(x)+\left(dL+\varepsilon\right)\max_{|x|\le R} V(x),
  \end{align*}
  and it follows from Gronwall's inequality that
  \begin{align*}
   \mb{E}_{\pi_\beta}\left[V(X_t)\right]=P_tV(x)\le V(x) e^{-\varepsilon t}+ \left(1-e^{-\varepsilon t}\right)  \left(dL+\varepsilon\right)\max_{|x|\le R} V(x).
  \end{align*}
  We hence have that $\mb{E}_{\pi_\beta}\left[V(X)\right]\le \left(dL+\varepsilon\right)\max_{|x|\le R} V(x)$ by taking $t\to\infty$. For the other expectation, we have
\begin{align*}
    \mb{E}_{\pi_\beta} \left[ |\nabla V(X)|^2 \right] &=Z_\beta^{-1} \int_{\mb{R}^d} \langle \nabla V(x),V(x)^{-\beta} \nabla V(x) \rangle dx\\
    &=-Z_\beta^{-1} \int_{\mb{R}^d} V(x)\nabla \cdot \left( V(x)^{-\beta} \nabla V(x) \right) dx \\
    &=\beta \mb{E}_{\pi_\beta} \left[ |\nabla V(X)|^2 \right]-  Z_\beta^{-1} \int_{\mb{R}^d} \Delta V(x) V(x)^{-(\beta-1)} dx.
\end{align*}
The above identity implies
\begin{align*}
    \mb{E}_{\pi_\beta} \left[ |\nabla V(X)|^2 \right] &=(\beta-1)^{-1} Z_{\beta}^{-1} \int_{\mb{R}^d} \Delta V(x) V(x)^{-(\beta-1)}dx\\
    &\le (\beta-1)^{-1} Z_{\beta}^{-1} \int_{\mb{R}^d} \text{trace}(\nabla^2 V(x)) V(x)^{-(\beta-1)}dx \\
    &\le (\beta-1)^{-1} Z_{\beta}^{-1} d L \int_{\mb{R}^d}  V(x)^{-(\beta-1)}dx \\
    &=\frac{d L}{\beta-1} \mb{E}_{\pi_\beta}\left[V(X)\right] \\
    &\le \frac{d L\left(dL+\varepsilon\right)}{\beta-1} \max_{|x|\le R} V(x).
\end{align*}
\end{proof}
\subsection{General Case}\label{sec:gencase}
Next we discuss the general case where $\pi_\beta=Z_{\beta}^{-1}V^\beta$ and $V\in \mc{C}^2(\mb{R}^d)$ is positive such that there exist constants $\alpha,L>0$ and $\alpha I_d \preceq \nabla^2 V(x)\preceq L I_d $ for all $x\in \mb{R}^d$. Since $V$ is strongly convex, there is a unique $x^*\in \mb{R}^d$ such that $V(x)\ge V(x^*)>0$ for all $x\in \mb{R}^d$ and $\nabla V(x^*)=0$. Without loss of generality, we assume $x^*=0$.
\begin{proposition}\label{prop:expectation under general V} Let $\beta>\frac{d}{2}+1$. If $V\in \mc{C}^2(\mb{R}^d)$ is positive, $\alpha$-strongly convex and $L$-gradient Lipschitz, we have for any $r\in (0,\beta-\frac{d}{2}-1)$,
\begin{align}
    &\mb{E}_{\pi_\beta}\left[ V(X) \right]\le \left(\frac{L}{\alpha}\right)^{\frac{\frac{d}{2}}{\beta-\frac{d}{2}-r}} V(0) \left( \frac{\Gamma(\beta)\Gamma(r)}{\Gamma(\frac{d}{2}+r)\Gamma(\beta-\frac{d}{2})} \right)^{\frac{1}{\beta-\frac{d}{2}-r}}, \label{eq:expectation under general V} \\
    &\mb{E}_{\pi_\beta}\left[ |\nabla V(X)|^2 \right]\le \frac{d L}{\beta-1} \left(\frac{L}{\alpha}\right)^{\frac{\frac{d}{2}}{\beta-\frac{d}{2}-r}} V(0) \left( \frac{\Gamma(\beta)\Gamma(r)}{\Gamma(\frac{d}{2}+r)\Gamma(\beta-\frac{d}{2})} \right)^{\frac{1}{\beta-\frac{d}{2}-r}}. \label{eq:expectation under general V gradient V}
\end{align}
\end{proposition}
\begin{proof}\label{pf:expectation under general V} For any $r\in (0,\beta-\frac{d}{2}-1)$, we have
\begin{align*} 
    \mb{E}_{\pi_\beta}\left[ V(X) \right]&=  \frac{\int_{\mb{R}^d} V(x) V(x)^{-\beta} dx}{Z_{\beta}} =\frac{Z_{\beta-1}}{Z_{\beta}}\le \left(\frac{L}{\alpha}\right)^{\frac{\frac{d}{2}}{\beta-\frac{d}{2}-r}} V(0) \left( \frac{\Gamma(\beta)\Gamma(r)}{\Gamma(\frac{d}{2}+r)\Gamma(\beta-\frac{d}{2})} \right)^{\frac{1}{\beta-\frac{d}{2}-r}}.
\end{align*}
where the last inequality follows from Lemma \ref{lem:ratio between normalization coeffcients}. For the other expectation, we have
\begin{align*}
    \mb{E}_{\pi_\beta} \left[ |\nabla V(X)|^2 \right] &=Z_\beta^{-1} \int_{\mb{R}^d} \langle \nabla V(x),V(x)^{-\beta} \nabla V(x) \rangle dx\\
    &=-Z_\beta^{-1} \int_{\mb{R}^d} V(x)\nabla \cdot \left( V(x)^{-\beta} \nabla V(x) \right) dx \\
    &=\beta \mb{E}_{\pi_\beta} \left[ |\nabla V(X)|^2 \right]-  Z_\beta^{-1} \int_{\mb{R}^d} \Delta V(x) V(x)^{-(\beta-1)} dx.
\end{align*}
The above identity implies
\begin{align*}
    \mb{E}_{\pi_\beta} \left[ |\nabla V(X)|^2 \right] &=(\beta-1)^{-1} Z_{\beta}^{-1} \int_{\mb{R}^d} \Delta V(x) V(x)^{-(\beta-1)}dx\\
    &\le (\beta-1)^{-1} Z_{\beta}^{-1} \int_{\mb{R}^d} \text{trace}(\nabla^2 V(x)) V(x)^{-(\beta-1)}dx \\
    &\le (\beta-1)^{-1} Z_{\beta}^{-1} d L \int_{\mb{R}^d}  V(x)^{-(\beta-1)}dx \\
    &=\frac{d L}{\beta-1} \frac{Z_{\beta-1}}{Z_\beta} \\
    &\le \frac{d L}{\beta-1} \left(\frac{L}{\alpha}\right)^{\frac{\frac{d}{2}}{\beta-\frac{d}{2}-r}} V(0) \left( \frac{\Gamma(\beta)\Gamma(r)}{\Gamma(\frac{d}{2}+r)\Gamma(\beta-\frac{d}{2})} \right)^{\frac{1}{\beta-\frac{d}{2}-r}}.
\end{align*}
where the last inequality also follows from Lemma \ref{lem:ratio between normalization coeffcients}.
\end{proof}
\begin{remark}\label{rem:gamma function simplification} A ratio between Gamma functions appears in \eqref{eq:expectation under general V} and \eqref{eq:expectation under general V gradient V}. The ratio can be written explicitly via properties of Gamma functions. 
\begin{itemize}[leftmargin=0.15in]
\item When $d$ is an even number and $d=2k$ for some integer $k$,
\begin{align*}
    \frac{\Gamma(\beta)\Gamma(r)}{\Gamma(\frac{d}{2}+r)\Gamma(\beta-\frac{d}{2})}&=\frac{\Gamma(r)}{\Gamma(\frac{d}{2}+r)} \frac{\Gamma(\beta)}{\Gamma(\beta-\frac{d}{2})}= \frac{\Gamma(r)}{\Gamma(r) \prod_{i=1}^{k} (\frac{d}{2}+r-i) }  \frac{\Gamma(\beta-\frac{d}{2}) \prod_{i=1}^{k} (\beta-i) }{\Gamma(\beta-\frac{d}{2})}\\
    &=\frac{\prod_{i=1}^{k} (\beta-i)}{\prod_{i=1}^{k} (\frac{d}{2}+r-i)}\le \left( \frac{\beta-\frac{d}{2}}{r} \right)^{\frac{d}{2}},
\end{align*}
\item When $d$ is an odd number with $d=2k-1$ for some integer $k$,
\begin{align*}
    \frac{\Gamma(\beta)\Gamma(r)}{\Gamma(\frac{d}{2}+r)\Gamma(\beta-\frac{d}{2})}&=\frac{\Gamma(r)}{\Gamma(\frac{d}{2}+r)} \frac{\Gamma(\beta)}{\Gamma(\beta-\frac{d}{2})}=\frac{\Gamma(r)}{\Gamma(\frac{1}{2}+r) \prod_{i=1}^{k-1} (\frac{d}{2}+r-i)} \frac{\Gamma(\beta-\frac{d}{2}+\frac{1}{2}) \prod_{i=1}^{k-1} (\beta-i) }{\Gamma(\beta-\frac{d}{2})} \\
    &= \frac{\prod_{i=1}^{k-1} (\beta-i)}{\prod_{i=1}^{k-1} (\frac{d}{2}+r-i)} \frac{r^{-1}\Gamma(r+1)}{\Gamma(\frac{1}{2}+r)} \frac{\Gamma(\beta-\frac{d}{2}+\frac{1}{2})}{\Gamma(\beta-\frac{d}{2})}\\
    &\le \left( \frac{\beta-\frac{d}{2}+\frac{1}{2}}{r+\frac{1}{2}} \right)^{k-1} r^{-1} (1+r)^{\frac{1}{2}} \left (\beta-\frac{d}{2}+\frac{1}{2} \right )^{\frac{1}{2}} \\
    &\le \sqrt{\frac{1+r}{r}}\left( \frac{\beta-\frac{d}{2}}{r} \right)^{\frac{d}{2}},
\end{align*}
where the first inequality follows from Gautschi's inequality \citep{ismail1994inequalities}.
\end{itemize}
\end{remark}
\begin{remark}\label{rem:order of expectations} With Theorem \ref{prop:expectation under general V} and the upper bounds in Remark \ref{rem:gamma function simplification}, we can get the estimations for $\mb{E}_{\pi_\beta}\left[|\nabla V(X)|^2\right]$ and $\mb{E}_{\pi_\beta}\left[ V(X) \right]$: for any $r\in (0,\beta-\frac{d}{2}-1)$,
\begin{align}
    \mb{E}_{\pi_\beta}\left[ V(X) \right]&\le V(0)\left(\frac{L}{\alpha}\right)^{\frac{\frac{d}{2}}{\beta-\frac{d}{2}-r}} \left(\frac{1+r}{r}\right)^{\frac{1}{2(\beta-\frac{d}{2}-r)}} \left(\frac{\beta-\frac{d}{2}}{r}\right)^{\frac{\frac{d}{2}}{\beta-\frac{d}{2}-r}}, \label{eq:order of expectation under general V} \\
    \mb{E}_{\pi_\beta}\left[ |\nabla V(X)|^2 \right]&\le \frac{V(0) d L}{\beta-1} \left(\frac{L}{\alpha}\right)^{\frac{\frac{d}{2}}{\beta-\frac{d}{2}-r}} \left(\frac{1+r}{r}\right)^{\frac{1}{2(\beta-\frac{d}{2}-r)}} \left(\frac{\beta-\frac{d}{2}}{r}\right)^{\frac{\frac{d}{2}}{\beta-\frac{d}{2}-r}}. \label{eq:order of expectation under general V gradient V}
\end{align}
\end{remark}

\section{Zeroth-Order It\^o Discretization}\label{sec:zosetting}

While previously we consider the case when the gradient of the function $V$ is analytically available to us, we now consider the case when we have access only to the function evaluations. This setting is called the zeroth-order setting and has been recently examined in the context of complexity of sampling in the works of~\cite{dwivedi2019log,lee2021structured,roy2022stochastic}. In this setting, we construct an approximation to the gradient via zeroth-order information, i.e., function evaluations. For simplicity, we consider the case of obtaining exact function evaluations. Based on the Gaussian smoothing technique \citep{nesterov2017random,roy2022stochastic}, for any $x\in \mb{R}^d$, we define the zeroth order gradient estimator $g_{\sigma,,m}(x)$ as
\begin{align}\label{eq:zeroth order apprximator}
    g_{\sigma,m}(x):=\frac{1}{m}\sum_{i=1}^m \frac{V(x+\sigma u_i)-V(x)}{\sigma} u_i
\end{align}
where $u_i\sim \mc{N}(0,I_d)$ are assumed to be independent and identically distributed. The parameter $m$ is called the batch size parameter. Then the zeroth order algorithm to sample $\pi_\beta$ is given by
\begin{align}\label{eq:zeroth order alg}
    x_{k+1}=x_k-h(\beta-1)g_{\sigma,m}(x_k)+\sqrt{2V(x_k)} \xi_{k+1}
\end{align}
where $h>0$ is the step size and $\{\xi_{k+1}\}_{k=0}^\infty$ is a sequence of independent identically distributed standard Gaussian random vectors in $\mb{R}^d$. From \cite{balasubramanian2022zeroth} and \cite{roy2022stochastic}, we recall the following property of $g_{\sigma,m}$.
\begin{proposition}\label{prop:zeroth order approximator}{\cite[Section 8.1]{roy2022stochastic}} Assume $V$ is $L$-gradient Lipschitz. Define $\zeta_k=g_{\sigma,m}(x_k)-\nabla V(x_k)$ with $g_{\sigma,m}$ defined in \eqref{eq:zeroth order apprximator} and $\{x_k\}_{k=0}^\infty$ generated by \eqref{eq:zeroth order alg}. We have for any $k\ge 0$,
\begin{align}\label{eq:bias bound}
    \mb{E}\left[ |\mb{E}\left[ \zeta_k|x_k \right]|^2 \right]\le L^2 \sigma^2 d,
\end{align}
and
\begin{align}\label{eq:variance bound}
    \mb{E}\left[ |\zeta_k-\mb{E}\left[\zeta_k|x_k\right]|^2 \right]\le \frac{\sigma^2}{2m}L^2(d+3)^3+\frac{2(d+5)}{m}\mb{E}\left[|\nabla V(x_k)|^2\right].
\end{align}
\end{proposition}
\begin{theorem}\label{thm:zeroth order algorithm strongly convex V} 
Suppose $V$ is gradient-Lipschitz with parameter $L>0$ and satisfies Assumption \ref{ass:WPI strongly convex V} with $\delta$ in~\eqref{eq:delta}. Let $g_{\sigma,m}$ be as defined in \eqref{eq:zeroth order apprximator} and $(x_k)_{k=0}^\infty$ be generated from \eqref{eq:zeroth order alg} with $x_k\sim \nu_k$ for all $k\ge 0$. Then with the time step size 
\begin{align}\label{eq:zeroh}
h<\min\left\{ \frac{2\delta}{3(1+\delta)\alpha(\beta-1)},\frac{\alpha m \delta} {24(1+\delta)(\beta-1)(d+5)L^2},\frac{1}{4(\beta-1)L}\right\},
\end{align}
the decay of Wasserstein-2 distance along the Markov chain $(x_k)_{k=0}^\infty$ can be described by the following equation. For all $k\ge 1$,
\begin{align}\label{eq:zeroth order W2 decay}
    W_2(\nu_k,\pi_\beta)\le (1-A')^k W_2(\nu_0,\pi_\beta)+\frac{C'}{A'}+\frac{B'}{\sqrt{A'(2-A')}}.
\end{align}
with $A',B'$ and $C'$ given respectively in \eqref{eq:parameter A'}, \eqref{eq:parameter B'} and \eqref{eq:parameter C'}.
\end{theorem}

\begin{remark}
With Theorem \ref{thm:zeroth order algorithm strongly convex V}, we can study the iteration complexity to reach an $\varepsilon$-accuracy in Wasserstein-2 distance. In the following discussion, we focus on the dimension dependence and $\varepsilon$ dependence in the iteration complexity. When $\beta= \Theta(d)$ and $\alpha,L= \Theta(1)$, and when $h$ satisfies~\eqref{eq:zeroh}, we have
\begin{align*}
   & A' =  O\left(\delta d h\right), \qquad  \frac{C'}{A'} = O\left(\frac{(dh\mb{E}_{\pi_\beta}\left[V(X)\right])^{\frac{1}{2}}+d h \mb{E}_{\pi_\beta}\left[|\nabla V(X)|^2\right]^{\frac{1}{2}}+\sigma  d^{\frac{1}{2}}}{\delta}\right),\\
   & \frac{B'}{\sqrt{A'(2-A')}} = O\Bigg( \left(\frac{dh}{\delta} + \frac{d h^{\frac{1}{2}}}{(\delta m)^{\frac{1}{2}}} \right)\mb{E}_{\pi_\beta}\left[|\nabla V(X)|^2\right]^{\frac{1}{2}} + \frac{(dh)^{\frac{1}{2}}}{\delta}\mb{E}_{\pi_\beta}\left[V(X)\right]+\frac{\sigma d^2 h^{\frac{1}{2}}}{(\delta m)^{\frac{1}{2}}} \Bigg).
\end{align*}
To ensure $W_2(\nu_K,\pi_\beta)<\varepsilon$, we require that each of $$
(1-A')^K W_2(\nu_0,\pi_\beta),\qquad \frac{C'}{A'},\qquad \frac{B'}{\sqrt{A'(2-A')}},$$ is smaller than $\varepsilon/3$. 
Setting $\sigma=\varepsilon \delta/ \sqrt{d}$, and
\begin{align*}
    h &=O\left( \min\left\{ \frac{(\varepsilon \delta)^2}{ d} \mb{E}_{\pi_\beta}\left[V(X)\right]^{-1}, \frac{\varepsilon \delta}{ d}\mb{E}_{\pi_\beta}\left[|\nabla V(X)|^2\right]^{-\frac{1}{2}}, \frac{\varepsilon^2 \delta m}{d^2}\mb{E}_{\pi_\beta}\left[|\nabla V(X)|^2\right]^{-1} \right\} \right), 
\end{align*}
we hence obtain that the iteration complexity $K$ is of order
\begin{align}\label{eq:zeroth order iteration complexity}
        K  = \Tilde{O}\left( \max\left\{ \frac{1}{\varepsilon^{2} \delta^{3}} \mb{E}_{\pi_\beta}\left[V(X)\right], \frac{1}{\varepsilon \delta^{2}}\mb{E}_{\pi_\beta}\left[|\nabla V(X)|^2\right]^{\frac{1}{2}}, \frac{d}{ \varepsilon^{2} \delta^{2} m} \mb{E}_{\pi_\beta}\left[|\nabla V(X)|^2\right] \right\} \right).
\end{align}
The number of function evaluations is hence $mK$. 
\end{remark}

\section{Illustrative Examples}\label{subsec:exponential contractivity}

We now provide illustrative examples to highlight the implications of our results.

\subsection{Multivariate $t$-distribution: Large Degree of Freedom}\label{subsec:cauchy large dof} We first consider the isotropic multivariate $t$-distribution with the degrees of freedom being $d+2$. We choose $V(x)=1+|x|^2$, $\beta=d+1$ and $\pi_\beta(x)\propto V(x)^{-\beta}=(1+|x|^2)^{-(d+1)}$. With this choice of $V$ and $\beta$, $V$ satisfies Assumption \ref{ass:WPI strongly convex V} with $\alpha=2,\ C_V=2$, and $V$ is $L$-Lipschitz gradient with $L=2$. The constant $\delta$ in Theorem \ref{thm:strongly convex V W2 decay} becomes $\delta=1$. Furthermore, according to proposition \ref{prop:expectations under Cauchy}, $\mb{E}_{\pi_\beta}[V(X)]=2$ and $\mb{E}_{\pi_\beta}[|\nabla V(X)|^2]=4$.

\subsubsection{First order algorithm}
 According to Theorem \ref{thm:strongly convex V W2 decay} and
\eqref{eq:strongly convex V iteration complexity K}, to obtain $\epsilon$-accuracy in Wasserstein-2 distance, the iteration complexity is of order $\Tilde{O}(1/\epsilon^{2})$. With the same choice of $V$ and $\beta$, we check the conditions of Theorem 1 in \cite{li2019stochastic}. The diffusion \eqref{eq:diffusion strongly convex V} is $\alpha'$-uniformly dissipative with $\alpha'=d$ and the Euler discretization given in \eqref{eq:EM strongly convex V} has local deviation with order $(p_1,p_2)=(1,{3}/{2})$ and $(\lambda_1,\lambda_2)=(\Theta(d^5),\Theta(d^4))$. The detailed calculation for deriving the constants above is provided in Appendix~\ref{cauchyexample}. Hence, by Theorem 1 in \cite{li2019stochastic}, to reach an $\epsilon$-accuracy in Wasserstein-2 distance, the iteration complexity is of order $\Tilde{O}(d^3/\epsilon^{2})$. Hence, in comparison with the result in \cite{li2019stochastic}, we obtain a dimension-free iteration complexity to ensure an $\epsilon$-accuracy in Wasserstein-2 distance.

\subsubsection{Zeroth order algorithm} 
According to Theorem \ref{thm:zeroth order algorithm strongly convex V} and
\eqref{eq:zeroth order iteration complexity}, to obtain $\varepsilon$-accuracy in Wasserstein-2 distance, the iteration complexity is of order $\Tilde{O}\left( (1 \vee d/m)/\varepsilon^{2} \right)$. When $m=1$, the iteration complexity $K\sim \Tilde{O}(d/\varepsilon^{2})$ and the number of functions evaluations $mK$ is also of the same order $\Tilde{O}(d/\varepsilon^{2})$. If we choose the batch size $m=d$, we get a dimension independent iteration complexity $K\sim \Tilde{O}(1/\varepsilon^{2})$ but the number of function evaluations is of order $\Tilde{O}(d/\varepsilon^{2})$. Hence, we notice that in the case of multivariate $t$-distribution distributions with large degrees of freedom, the cost of estimating the gradient has an effect on the sampling complexities.

\subsection{Multivariate $t$-distribution: Small Degrees of Freedom}\label{subsec:cauchy small dof} We now consider the isotropic multivariate $t$-distribution with the degrees of freedom being $3$. We denote the corresponding density function by $\pi_\beta$. The exact number of 3 is chosen just for convenience; the results of this example apply to all cases where the degrees of freedom is \emph{strictly} larger than 2 which corresponds to the setting where the variance is finite. We choose $V(x)=1+|x|^2$, $\beta={(d+3)}/{2}$ and $\pi_\beta(x)\propto V(x)^{-\beta}=(1+|x|^2)^{-(d+3)/2}$. With the above choice of $V$ and $\beta$, $V$ satisfies Assumption \ref{ass:WPI strongly convex V} with $\alpha=2,\ C_V=2$ and $V$ is $L$-Lipschitz gradient with $L=2$. Hence, the constant $\delta$ in Theorem \ref{thm:strongly convex V W2 decay} is given by $\delta=1/d$. According to Proposition \ref{prop:expectations under Cauchy}, $\mb{E}_{\pi_\beta}[V(X)]=d+1$ and $\mb{E}_{\pi_\beta}[|\nabla V(X)|^2]=4d$.

\subsubsection{First order algorithm} According to Theorem \ref{thm:strongly convex V W2 decay} and
\eqref{eq:strongly convex V iteration complexity K}, to obtain $\epsilon$-accuracy in Wasserstein-2 distance, the iteration complexity is of order $\Tilde{O}(d^4/\epsilon^{2})$. With the same choice of $V$ and $\beta$, we check the conditions of Theorem 1 in \cite{li2019stochastic}. The diffusion \eqref{eq:diffusion strongly convex V} is $\alpha'$-uniformly dissipative with $\alpha'=1$ and the Euler discretization given in \eqref{eq:EM strongly convex V} has local deviation with order $(p_1,p_2)=(1,{3}/{2})$ and $(\lambda_1,\lambda_2)=(\Theta(d^5),\Theta(d^4))$. The detailed calculation for deriving the constants is provided in Appendix~\ref{cauchyexample}. Hence, according to Theorem 1 in \cite{li2019stochastic}, to reach an $\epsilon$-accuracy in Wasserstein-2 distance, the iteration complexity is of order $\Tilde{O}(d^6/\epsilon^{2})$. Even in this extremely heavy-tail case (i.e., only the variance exists), to ensure an $\epsilon$-accuracy in Wasserstein-2 distance, we can obtain an iteration complexity with polynomial dimension dependence. Furthermore, in comparison to \cite{li2019stochastic}, our analysis helps to decrease the dimension exponent by a factor of $2$.

\subsubsection{Zeroth order algorithm}\label{sec:cauchy small dof zeroth order}
According to Theorem \ref{thm:zeroth order algorithm strongly convex V} and
\eqref{eq:zeroth order iteration complexity}, to obtain $\varepsilon$-accuracy in Wasserstein-2 distance, the iteration complexity is of order $\Tilde{O}\left( \max\{ d^4/\varepsilon^{2},~d^{\frac{5}{2}}/\varepsilon,~d^4/\varepsilon^{2}m\} \right)$. Hence, we have that for any batch size $m$, the iteration complexity $K= \Tilde{O}(d^4/\varepsilon^{2})$. Picking $m=1$, the number of function evaluations are of the same order, i.e.,  $mK= \Tilde{O}(d^4/\varepsilon^{2})$. 

\begin{remark}
The example discussed in Section \ref{sec:cauchy small dof zeroth order} highlights the following important observation: Choosing a large batch size does not improve the iteration complexity. To explain this, we understand both \eqref{eq:EM strongly convex V} and \eqref{eq:zeroth order alg} as approximation to the continuous dynamics \eqref{eq:diffusion strongly convex V}. For the first-order algorithm, the error of the approximation only comes from the Euler-Maruyama discretization. For the zeroth-order algorithm, the error of the approximation comes from both the Euler-Maruyama discretization and the zeroth-order gradient estimate=. When the error from the Euler-Maruyama discretization dominates, the optimal batch size is always $1$ and the oracle complexity of the zeroth order algorithm is the same as the iteration complexity for the first-order algorithm. When the error from the zeroth-order gradient estimate dominates, we need to choose a large batch size depending on $d$ so that the iteration complexity for the zeroth-order algorithm is the same as the iteration complexity for the first-order algorithm while the zeroth-order oracle complexity is of order $m$-times larger.
\end{remark}

\section{Further Results and Additional Insights on Assumptions}\label{sec:further} 

In Section \ref{sec:weighted Poincare inequality}, we provide sufficient conditions on $V$ such that when $\beta>d$, $\pi_\beta\propto V^{-\beta}$ satisfies the weighted Poincar\'{e} inequality with weight $V$. In this section, we relax the conditions in Section \ref{sec:weighted Poincare inequality} by introducing the following assumptions.
\begin{assump}\label{ass:WPI on V small beta} The function $V:\mb{R}^d\to (0,\infty)$ is twice continuously differentiable and $V$ satisfies 
\begin{itemize}[noitemsep]
    \item [(1)] $\nabla^2 V(x)$ is invertible for all $x\in \mb{R}^d$.
    \item [(2)] There exists $\gamma\in \left(0,\frac{\beta}{d+2}\right]$, such that 
    \begin{align*}
    \sup_{x\in \mb{R}^d}\lv V(x)^{\gamma-1} \left(\nabla^2 V_\gamma\right)^{-1}(x) \rv_2\le C_V(\gamma),
    \end{align*}
    where $V_\gamma:=V^{\gamma}$ and $C_V(\gamma)$ is a positive constant depending on $\gamma$.
\end{itemize}
\end{assump}
\begin{lemma}\label{lem:WPI small beta} Under Assumption \ref{ass:WPI on V small beta}, for any smooth function $\phi\in L^2(\pi_\beta)$,
\begin{align}\label{eq:WPI small beta}
    Var_{\pi_\beta}(\phi)\le \cwpi \int_{\mb{R}^d} |\nabla \phi(x)|^2 V(x) \pi_\beta(x)dx, \quad\text{with}\quad\cwpi=C_V(\gamma)\left(\frac{\beta}{\gamma}-1\right)^{-1}.
\end{align}
\end{lemma}

\begin{proof} First we define $V_{\gamma}:=V^{\gamma}$. Choose $\beta'=\beta-2\gamma$. For $\pi_{\beta'}\propto V^{-\beta'}$, we can write it as $\pi_{\beta'}\propto {V_{\gamma}}^{-a}$ with 
\begin{align*}
    a=\frac{\beta'}{\gamma}=\frac{\beta-2\gamma}{\gamma}\ge d,
\end{align*}
where the inequality follows from the fact that $\gamma\in \left(0,\frac{\beta}{d+2}\right]$. Therefore we can apply Theorem \ref{thm:2.3} to $\pi_{\beta'}\propto {V_{\gamma}}^{-a}$ and get for any smooth, $\pi_{\beta'}$-square integrable function $g$ with $\mb{E}_{\pi_{\beta'}}[g(X)]=0$ and $G=V_{\gamma}g$,
\begin{align}\label{eq:WPI small beta  step 1}
    (a+1)\int_{\mb{R}^d} g(x)^2 \pi_{\beta'}(x)dx \le \int_{\mb{R}^d} \frac{\langle (\nabla^2 V_{\gamma})^{-1}(x)\nabla G(x),\nabla G(x) \rangle}{V_{\gamma}(x)} \pi_{\beta'}(x)dx.
\end{align}
Since $\beta'=\beta-2\gamma$, \eqref{eq:WPI small beta  step 1} is equivalent to
\begin{align}\label{eq:WPI small beta  step 2}
    (a+1)\int_{\mb{R}^d} \frac{|G(x)|^2}{V(x)} V(x)^{-(\beta-1)} dx \le \int_{\mb{R}^d} \langle (\nabla^2 V_{\gamma})^{-1}(x)\nabla G(x),\nabla G(x) \rangle V(x)^{-(\beta'+\gamma)} dx.
\end{align}
Under Assumption \ref{ass:WPI on V small beta}, we have
\begin{align*}
    &\int_{\mb{R}^d} \langle (\nabla^2 V_{\gamma})^{-1}(x)\nabla G(x),\nabla G(x) \rangle V(x)^{-(\beta'+\gamma)} dx\\
    &\le  C_V(\gamma) \int_{\mb{R}^d}|\nabla G(x)|^2 V(x)^{1-\gamma}V(x)^{-(\beta'+\gamma)} dx \\
    &=C_V(\gamma) \int_{\mb{R}^d} |\nabla G(x)|^2 V(x)^{-(\beta-1)} dx,
\end{align*}
where the last identity follows from the fact that $\beta'=\beta-2\gamma$. Along with \eqref{eq:WPI small beta  step 2}, we get
\begin{align}\label{eq:WPI small beta step 3}
    (a+1)\int_{\mb{R}^d} \frac{|G(x)|^2}{V(x)} V(x)^{-(\beta-1)} dx \le C_V(\gamma) \int_{\mb{R}^d} |\nabla G(x)|^2 V(x)^{-(\beta-1)} dx.
\end{align}
Since $G=V^\gamma g$, $G$ is smooth, $\pi_\beta$-square integrable and $\mb{E}_{\pi_{\beta-\gamma}}[G(X)]=0$. For any $\pi_\beta$-square integrable $\phi$, let $G=\phi-\mb{E}_{\pi_{\beta-\gamma}}[\phi(X)]$ and we get
\begin{align}\label{eq:WPI small beta step 4}
    \int_{\mb{R}^d} |\phi(x)-\mb{E}_{\pi_{\beta-\gamma}}[\phi(X)]|^2 \pi_{\beta}(x)dx \le \frac{C_V(\gamma)}{a+1} \int_{\mb{R}^d} |\nabla \phi(x)|^2 V(x) \pi_{\beta}(x)dx.
\end{align}
Therefore for any smooth, $\pi_\beta$-square integrable $\phi$, 
\begin{align*}
    Var_{\pi_\beta} (\phi)&= \inf_{c\in \mb{R}} \int_{\mb{R}^d} |\phi(x)-c|^2 \pi_{\beta}(x)dx \le \frac{C_V(\gamma)}{a+1} \int_{\mb{R}^d} |\nabla \phi(x)|^2 V(x) \pi_{\beta}(x)dx,
\end{align*}
which is equivalent to \eqref{eq:WPI small beta} with $\cwpi=\frac{C_V(\gamma)}{a+1}=C_V(\gamma)\left(\frac{\beta}{\gamma}-1\right)^{-1}$.
\end{proof}
\begin{remark}\label{rem:cauchywpi}
Lemma \ref{lem:WPI small beta} can be applied to the class of multivariate $t$-distributions with $V(x)=1+|x|^2$. When $\beta\in \left(\frac{d+2}{2},d\right]$, with the choice of $\gamma=\frac{\beta}{d+2}$, Assumption \ref{ass:WPI on V small beta} holds with $$C_V(\gamma)=\frac{(d+2)^2}{2\beta(2\beta-d-2)}.$$ 
Hence, Lemma \ref{lem:WPI small beta} implies that the multivariate $t$-distribution with degree of freedom $\nu\in(2,d]$ satisfies the weighted Poincar\'{e} inequality with weight $1+|x|^2$ and with $$\cwpi=\frac{(d+2)^2}{\nu(d+1)(d+\nu) }.$$
The detailed calculation for deriving the above mentioned constants is provided in Appendix~\ref{sec:cauchy application}.
\end{remark}

As an immediate consequence of Lemma~\ref{lem:WPI small beta}, we have the following $\chi^2$ convergence result for~\eqref{eq:diffusion strongly convex V}.

\begin{proposition}\label{prop:chi square cts small beta} Under Assumption \ref{ass:WPI on V small beta}, with $(X_t)$ satisfying \eqref{eq:diffusion strongly convex V} with $\rho_t$ being the distribution of $X_t$, we have
\begin{align}
    \chi^2(\rho_t|\pi_\beta)\le \exp\left( -C_V(\gamma)^{-1}\left(\frac{\beta}{\gamma}-1\right) t \right)\chi^2(\rho_0|\pi_\beta).
\end{align}
\end{proposition}
For the case of multivariate $t$-distributions, Proposition~\ref{prop:chi square cts small beta} allows us to show exponential convergence of~\eqref{eq:diffusion strongly convex V} in the $\chi^2$ divergence with smaller degrees of freedom (and hence heavier tails) compared to Proposition~\ref{prop:chi square cts strongly convex V}. 

\subsection{Relationship between Lemma \ref{cor:WPI strongly convex V} and Lemma \ref{lem:WPI small beta}}
 The result in Lemma \ref{lem:WPI small beta} complements that in Lemma \ref{cor:WPI strongly convex V}. It can be used to study the WPI for $\pi_\beta$ when $\beta\le d$. In particular, when $\beta\le d$, if $\pi_\beta\propto V^{-\beta}$ and $V$ satisfies Assumption \ref{ass:WPI strongly convex V} with $C_V\in (0,\frac{d+2}{d+2-\beta})$, then V satisfies Assumption \ref{ass:WPI on V small beta}. Therefore $\pi_\beta$ satisfies the WPI. In Proposition \ref{prop:comparison V2 V3}, this relation is proved formally. 

\begin{proposition}\label{prop:comparison V2 V3} 
When $\beta\le d$, if Assumption \ref{ass:WPI strongly convex V} holds with $C_V\in (0,\frac{d+2}{d+2-\beta})$, then Assumption \ref{ass:WPI on V small beta} holds.
\end{proposition}
\begin{proof}\label{pf:comparison V2 V3}
First $\nabla^2 V$ is invertible because $\nabla^2 V \succeq \alpha I_d$. Next we show that there exists $\gamma\in (0,\frac{\beta}{d+2}]$ such that $\lv V(x)^{\gamma-1} (\nabla^2 V_\gamma)^{-1}(x) \rv_2\le C_V(\gamma)$ for all $x\in \mb{R}^d$. It is equivalent to showing that there exists $\gamma\in (0,\frac{\beta}{d+2}]$ such that $\lv V(x)^{1-\gamma} (\nabla^2 V_\gamma)(x) \rv_2>0$ for all $x\in \mb{R}^d$. From the calculations in Section \ref{sec:cauchy application}, we have
\begin{align*}
    \nabla^2 V_\gamma(x)=\gamma V(x)^{\gamma-1}\left( (\gamma-1)V(x)^{-1}\nabla V(x)^T\nabla V(x)+\nabla^2 V(x) \right).
\end{align*}
Therefore
\begin{align*}
    V(x)^{1-\gamma} (\nabla^2 V_\gamma)(x)&=\gamma\left( \nabla^2 V(x)-(1-\gamma)V(x)^{-1}\nabla V(x)^T\nabla V(x) \right) \\
    &\succeq \alpha\gamma\left( 1-(1-\gamma)C_V \right) I_d, 
\end{align*}
where the inequality follows from Assumption \ref{ass:WPI strongly convex V}. Last we show that there exists $\gamma\in (0,\frac{\beta}{d+2}]$ such that $1-(1-\gamma)C_V>0$. Note that
\begin{align*}
    1-(1-\gamma)C_V>0~\implies \gamma>1-\frac{1}{C_V}.
\end{align*}
Since $C_V\in\left(0,\frac{d+2}{d+2-\beta}\right)$, we have that
\begin{align*}
    1-\frac{1}{C_V}< \frac{\beta}{d+2}
\end{align*}
Therefore there exists a constant $\gamma\in \left(0,\frac{\beta}{d+2}\right]$ such that $\lv V(x)^{1-\gamma} (\nabla^2 V_\gamma)(x) \rv_2>0$ for all $x\in \mb{R}^d$.
\end{proof}
\subsection{Relationship between Theorem \ref{thm:strongly convex V W2 decay} and Proposition \ref{prop:chi square cts small beta}}

Proposition \ref{prop:chi square cts small beta} studies the convergence of the continuous dynamics \eqref{eq:diffusion strongly convex V} while Theorem \ref{thm:strongly convex V W2 decay} studies the convergence of the discretization \eqref{eq:EM strongly convex V}. The conditions in Theorem \ref{thm:strongly convex V W2 decay} can be shown to imply conditions in proposition \ref{prop:chi square cts small beta}. In Proposition \ref{prop:chi square cts small beta} we only assume Assumption \ref{ass:WPI on V small beta}.
 In Theorem \ref{thm:strongly convex V W2 decay}, we assume (i) Assumption \ref{ass:WPI strongly convex V}, (ii) $\delta=\frac{\beta-1-\frac{1}{4}C_V d}{\frac{1}{4}C_V d}>0$, and (iii) $V$ is gradient Lipschitz. In the following proposition, we show that these three assumptions together imply Assumption \ref{ass:WPI on V small beta}.
 
 \begin{proposition}\label{prop:relation between thm2 and prop6} If Assumption \ref{ass:WPI strongly convex V} holds such that $\delta=\frac{\beta-1-\frac{1}{4}C_V d}{\frac{1}{4}C_V d}>0$ and $V$ is $L$-gradient 
 Lipschitz, then Assumption \ref{ass:WPI on V small beta} holds.
 \end{proposition}
 \begin{proof}[Proof of Proposition \ref{prop:chi square cts small beta}] Under Assumption \ref{ass:WPI strongly convex V} and $L$-gradient Lipschitzness assumption, we have that $V$ is `essential quadratic'. That is, assuming $V$ attains its global minimum at $x^*$, for all $x\in \mb{R}^d$,
\begin{align*}
    V(x^*)+\frac{\alpha}{2}|x-x^*|^2 \le V(x)\le V(x^*)+\frac{L}{2}|x-x^*|^2.
\end{align*}
Therefore for all $x\in \mb{R}^d$,
\begin{align*}
    \frac{|\nabla V(x)|^2}{V(x)}&\le \frac{L^2|x-x^*|^2}{V(x^*)+\frac{\alpha}{2}|x-x^*|^2}\le \frac{2L^2}{\alpha},
\end{align*}
which implies that Assumption \ref{ass:WPI strongly convex V}-(2) is satisfied with $C_V=\frac{2L^2}{\alpha^2}$. Furthermore, 
\begin{align*}
    V(x)^{1-\gamma}(\nabla^2 V_\gamma)(x) \succeq \alpha\gamma\left(1-(1-\gamma)C_V\right) I_d =\alpha\gamma\left(1-2(1-\gamma)\frac{L^2}{\alpha^2}\right) I_d.
\end{align*}
The condition $\delta=\frac{\beta-1-\frac{1}{4}C_V d}{\frac{1}{4}C_V d}>0$ is equivalent to the condition $\beta>\frac{L^2}{2\alpha^2}d+1$. Notice that for all $d\ge 1$, we have
\begin{align*}
    \left(1-\frac{\alpha^2}{2L^2}\right)\left(d+2\right)< \frac{L^2}{2\alpha^2}d+1
\end{align*}
Therefore for any $$\beta>\frac{L^2}{2\alpha^2}d+1> \left(1-\frac{\alpha^2}{2L^2}\right)(d+2),$$ we can choose $\gamma=\frac{\beta}{d+2}$ and obtain
\begin{align*}
     V(x)^{1-\gamma}(\nabla^2 V_\gamma)(x) &\succeq \frac{2 L^2 \beta}{\alpha( d+2)}\left(\frac{\alpha^2}{2L^2}+\frac{\beta}{d+2}-1\right)I_d \\
     &=\frac{2 L^2 \beta}{\alpha( d+2)^2}\left(\beta-\left(1-\frac{\alpha^2}{2L^2}\right)(d+2)\right)I_d
\end{align*}
Therefore Assumption \ref{ass:WPI on V small beta}-(2) is satisfied with $\gamma=\beta/(d+2)$ and 
$$
C_V(\gamma)=\frac{\alpha(d+2)^2}{2L^2\beta}\left( \beta-\left(1-\frac{\alpha^2}{2L^2}\right)(d+2) \right)^{-1}>0.
$$
The proof is now complete because Assumption \ref{ass:WPI on V small beta}-(1) is automatically satisfied under Assumption \ref{ass:WPI strongly convex V}.
 
 \end{proof}

\section{Proofs of the Main Results}\label{sec:proof}

\subsection{Proofs of Theorem \ref{thm:strongly convex V W2 decay} and Theorem \ref{thm:zeroth order algorithm strongly convex V}} 

In this section, we provide the proof of Theorem \ref{thm:strongly convex V W2 decay} and Theorem \ref{thm:zeroth order algorithm strongly convex V} via mean square analysis. We first start with the following intermediate result.
\begin{proposition}\label{prop:moment difference continuous}  Let $(X_t)_{t\ge 0}$ follow \eqref{eq:diffusion strongly convex V} with $X_t\sim \rho_t$ for all $t\ge 0$. If $V$ is gradient Lipschitz with parameter $L$, then we have 
\begin{align}\label{eq:moment difference continuous}
    \begin{aligned}
    \mb{E}\left[ |X_t-X_0|^2 \right]&\le  4 \left[ (\beta-1)^2t^2 \mb{E}\left[|\nabla V(X_0)|^2\right] +td\mb{E}\left[V(X_0)\right] \right] \\ &\qquad \exp\left( 4(\beta-1)^2 L^2 t^2+d(\beta-1)L^2 t^2+2d L t \right).
    \end{aligned}
\end{align}
\end{proposition}
\begin{proof}[Proof of Proposition~\ref{prop:moment difference continuous}] According to \eqref{eq:diffusion strongly convex V},
\begin{align*}
    \mb{E}[|X_t-X_0|^2]\le 2(\beta-1)^2\mb{E} \left[ \left|\int_0^t \nabla V(X_s) ds\right|^2 \right]+4d\mb{E}\left[ \int_0^t V(X_s)ds \right],
\end{align*}
where 
\begin{align}\label{eq:strongly convex V bound 4}
    \mb{E} \left[ \left|\int_0^t \nabla V(X_s) ds\right|^2 \right] &\le 2\mb{E}\left[ \left(\int_0^t |\nabla V(X_s)-\nabla V(X_0)| ds \right)^2\right]+2 \mb{E}\left[ \left(\int_0^t |\nabla V(X_0)| ds \right)^2\right] \nonumber \\
    &\le 2t\mb{E}\left[  \int_0^t |\nabla V(X_s)-\nabla V(X_0)|^2 ds \right]+2t \mb{E}\left[ \int_0^t |\nabla V(X_0)|^2 ds \right] \nonumber \\
    &\le 2 L^2 t  \int_0^t \mb{E}\left[|X_s-X_0|^2\right] ds +2t^2 \mb{E}\left[ |\nabla V(X_0)|^2\right],
\end{align}
and
\begin{align}\label{eq:strongly convex V bound 5}
    &~~~\mb{E}\left[ \int_0^t V(X_s)ds \right]\\
    &\le \mb{E}\left[ \int_0^t V(X_0)+\langle \nabla V(X_0),X_s-X_0 \rangle+\frac{L}{2}|X_s-X_0|^2  ds \right] \nonumber\\
    &=t \mb{E}\left[ V(X_0) \right]+\frac{L}{2}\mb{E}\left[ \int_0^t |X_s-X_0|^2 ds\right]-(\beta-1)\mb{E}\left[ \int_0^t \int_0^s \langle \nabla V(X_0),\nabla V(X_u) \rangle duds \right] \nonumber \\
    &\le t \mb{E}\left[ V(X_0) \right]+\frac{L}{2}\mb{E}\left[ \int_0^t |X_s-X_0|^2 ds\right] -\frac{(\beta-1)t^2}{2}\mb{E}\left[|\nabla V(X_0)|^2\right] \nonumber \\
    &\quad -(\beta-1)\mb{E}\left[ \int_0^t \int_0^s \langle \nabla V(X_0),\nabla V(X_u)-\nabla V(X_0) \rangle duds \right] \nonumber \\
    &\le  t \mb{E}\left[ V(X_0) \right]+\frac{L}{2}\mb{E}\left[ \int_0^t |X_s-X_0|^2 ds\right] -\frac{(\beta-1)t^2}{2}\mb{E}\left[|\nabla V(X_0)|^2\right] \nonumber \\
    &\quad +\frac{(\beta-1)t^2}{2}\mb{E}\left[|\nabla V(X_0)|^2\right]+\frac{\beta-1}{4} \mb{E}\left[ \int_0^t \int_0^s |\nabla V(X_u)-\nabla V(X_0)|^2 duds \right] \nonumber \\
    &\le t \mb{E}\left[ V(X_0) \right]+\frac{L}{2}\mb{E}\left[ \int_0^t |X_s-X_0|^2 ds\right]+\frac{(\beta-1)L^2}{4}\mb{E}\left[ \int_0^t \int_0^s |X_u-X_0|^2 duds \right] \nonumber\\
    &\le t \mb{E}\left[ V(X_0) \right]+\left( \frac{L}{2} +\frac{(\beta-1)L^2 t}{4} \right)\mb{E}\left[ \int_0^t |X_s-X_0|^2 ds\right].
\end{align}
With \eqref{eq:strongly convex V bound 4} and \eqref{eq:strongly convex V bound 5}, we get
\begin{align*}
    \mb{E}[|X_t-X_0|^2]&\le \int_0^t  \left[ 4(\beta-1)^2 {L}^2 t+2d L+d(\beta-1)L^2 t \right]\mb{E}\left[|X_s-X_0|^2\right] ds +4dt\mb{E}\left[V(X_0)\right] \\
    &\quad +  4(\beta-1)^2t^2 \mb{E}\left[|\nabla V(X_0)|^2\right].
\end{align*}
By Gronwall's inequality, we hence have 
\begin{align*}
    \mb{E}[|X_t-X_0|^2]&\le 4 \left[ (\beta-1)^2t^2 \mb{E}\left[|\nabla V(X_0)|^2\right]+dt\mb{E}\left[V(X_0)\right] \right] \\
    &\qquad \exp\left( 4(\beta-1)^2 {L}^2 t^2+d(\beta-1)L^2 t^2+2d L t \right).
\end{align*}
\end{proof}
Based on the above proposition, we now prove Theorem~\ref{thm:strongly convex V W2 decay} below.\\

\begin{proof}[Proof of theorem \ref{thm:strongly convex V W2 decay}]
We perform mean square analysis to \eqref{eq:EM strongly convex V}. Let $(X_t)_{t\ge 0}$ follow \eqref{eq:diffusion strongly convex V} with $X_0\sim \pi_\beta$. Since $\pi_\beta$ is the unique stationary distribution to \eqref{eq:diffusion strongly convex V}, $X_t\sim\pi_\beta$ for all $t\ge 0$. With \eqref{eq:EM strongly convex V}, we can calculate the difference between $X_h$ and $x_1$,
\begin{align*}
    X_h-x_1&=X_0-\int_0^h (\beta-1)\nabla V(X_t)dt+\int_0^t \sqrt{2V(X_t)} dB_t -\left( x_0-(\beta-1)h y_0+\sqrt{2hV(x_0)} \xi_1 \right) \\
    &=(X_0-x_0)-(\beta-1)h\left( \nabla V(X_0)-\nabla V(x_0) \right)-\int_0^h (\beta-1)\left(\nabla V(X_t)-\nabla V(X_0)\right)dt\\
    &\quad \int_0^h \left(\sqrt{2V(X_t)}-\sqrt{2V(x_0)} \right)dB_t\\
    &:= U_1+U_2+U_3,
\end{align*}
where 
\begin{align}
    U_1&:=(X_0-x_0)-(\beta-1)h\left( \nabla V(X_0)-\nabla V(x_0) \right), \label{eq:U1} \\
    U_2&:=-\int_0^h (\beta-1)\left(\nabla V(X_t)-\nabla V(X_0)\right)dt, \label{eq:U2}\\
    U_3&:=\int_0^h \left(\sqrt{2V(X_t)}-\sqrt{2V(x_0)} \right)dB_t .\label{eq:U3}
\end{align}
Therefore according to triangle inequality,
\begin{align*}
    \mb{E}[|X_h-x_1|^2|]^{\frac{1}{2}}&\le \mb{E}[|U_1+U_3|^2]^{\frac{1}{2}}+\mb{E}[|U_2|^2]^{\frac{1}{2}}.
\end{align*}
Since $U_1$ is adapted to $\mc{F}_0$ and $\mb{E}[U_3|\mc{F}_0]=0$, we get
\begin{align*}
    \mb{E}[|U_1+U_3|^2|\mc{F}_0]&=|U_1|^2+\mb{E}[|U_3|^2|\mc{F}_0]\\
    &=|(X_0-x_0)-(\beta-1)h\left( \nabla V(X_0)-\nabla V(x_0) \right)|^2\\
    &\qquad+\mb{E}\left[\int_0^h \lv \sqrt{2V(X_t)} I_d-\sqrt{2V(x_0)} I_d \rv_F^2 dt |\mc{F}_0 \right].
\end{align*}
Since $V$ is $\alpha$-strongly convex and $L$-gradient Lipschitz, it satisfies
\begin{align*}
    \langle X_0-x_0, \nabla V(X_0)-\nabla V(x_0) \rangle\ge \frac{\alpha L}{\alpha+L}|X_0-x_0|^2+\frac{1}{\alpha+L}|\nabla V(X_0)-\nabla V(x_0)|^2.
\end{align*}
Therefore when $h\le \frac{2}{(\beta-1)(\alpha+L)}$,
\begin{align}\label{eq:strongly convex V U1}
    &|(X_0-x_0)-(\beta-1)h\left( \nabla V(X_0)-\nabla V(x_0) \right)|^2 \nonumber\\ 
   =&|X_0-x_0|^2-2(\beta-1)h\langle X_0-x_0,\nabla V(X_0)-\nabla V(x_0) \rangle+(\beta-1)^2 h^2 |\nabla V(X_0)-\nabla V(x_0)|^2 \nonumber\\
   \le & \left( 1-\frac{2(\beta-1)\alpha L h}{\alpha+L} \right)|X_0-x_0|^2 +(\beta-1)h \left( (\beta-1)h-\frac{2}{\alpha+L} \right)|\nabla V(X_0)-\nabla V(x_0)|^2 \nonumber \\
   \le & \left( 1-(\beta-1)\alpha h \right)^2 |X_0-x_0|^2.
\end{align}
Meanwhile, for arbitrary $r>0$, we have
\begin{align*}
      & \mb{E}\left[\int_0^h \lv \sqrt{2V(X_t)}-\sqrt{2V(x_0)} \rv_F^2 dt  \right] \\
    = & d \mb{E} \left[\int_0^h |\sqrt{2V(X_t)}-\sqrt{2V(x_0)}| ^2 dt \right] \\
    \le & d \left( h\left(\sqrt{2V(X_0)}-\sqrt{2V(x_0)}\right)^2 +\mb{E}\left[\int_0^h \left|\sqrt{2V(X_t)}-\sqrt{2V(X_0)}\right|^2 dt\right]\right) \\
    & + 2d|\sqrt{2V(X_0)}-\sqrt{2V(x_0)}| h^{\frac{1}{2}} \mb{E}\left[\int_0^h \left|\sqrt{2V(X_t)}-\sqrt{2V(X_0)}\right|^2 dt  \right] \\
    \le & d(1+r) h\left(\sqrt{2V(X_0)}-\sqrt{2V(x_0)}\right)^2+ d(1+r^{-1})\mb{E}\left[\int_0^h \left|\sqrt{2V(X_t)}-\sqrt{2V(X_0)}\right|^2 dt  \right].
\end{align*}
Notice that under Assumption~\ref{ass:WPI strongly convex V}, we have $$|\nabla (\sqrt{2V(x)})|=\frac{\sqrt{2}|\nabla V(x)|}{2\sqrt{V(x)}}\le \frac{\sqrt{2\alpha C_V}}{2},$$ 
for all $x\in\mb{R}^d$. Therefore
\begin{align}\label{eq:strongly convex V bound 1}
    (\sqrt{2V(X_0)}-\sqrt{2V(x_0)})^2 \le \frac{\alpha C_V}{2} |X_0-x_0|^2,
\end{align}
and 
\begin{align}\label{eq:strongly convex V bound 2}
    \int_0^h |\sqrt{2V(X_t)}-\sqrt{2V(X_0)}|^2 dt&\le \frac{\alpha C_V}{2}\int_0^h |X_t-X_0|^2 dt.
\end{align}
With \eqref{eq:strongly convex V bound 1} and \eqref{eq:strongly convex V bound 2}, we get
\begin{align}\label{eq:strongly convex V bound 3}
\begin{aligned}
    \mb{E}[\int_0^h \lv \sqrt{2V(X_t)}-\sqrt{2V(x_0)} \rv_F^2 dt  ]& \le \frac{\alpha C_V d h(1+r)}{2}\mb{E}[|X_0-x_0|^2] \\ 
    &\qquad +\frac{\alpha C_V d(1+r^{-1})}{2}\int_0^h \mb{E}[|X_t-X_0|^2] dt.
\end{aligned}
\end{align}
Next we apply Proposition \ref{prop:moment difference continuous} to $\mb{E}[|X_t-X_0|^2]$. 
In particular,  when 
$$t\in[0,h]\quad\text{and}\quad h<\frac{1}{4(\beta-1)L},$$
we have
\begin{align}\label{eq:strongly convex V moment bound}
    \mb{E}[|X_t-X_0|^2]&\le \left( 4dt\mb{E}\left[V(X_0)\right]+4(\beta-1)^2 t^2 \mb{E}\left[|\nabla V(X_0)|^2\right] \right) \exp(1)\nonumber\\
    &\le 12dt\mb{E}\left[V(X_0)\right]+12(\beta-1)^2 t^2 \mb{E}\left[|\nabla V(X_0)|^2\right].
\end{align}
Combining \eqref{eq:strongly convex V bound 3} and \eqref{eq:strongly convex V moment bound}, when $h<\frac{1}{4(\beta-1)L}$, we have that
\begin{align}\label{eq:strongly convex V U3}
     &\mb{E}[\int_0^h \lv \sqrt{2V(X_t)}-\sqrt{2V(x_0)} \rv_F^2 dt  ] \nonumber \\
     \le & \frac{1}{2}\alpha C_V d (1+r) h\mb{E}[|X_0-x_0|^2]\\ 
     &\quad +6\alpha C_V d(1+r^{-1})\int_0^h \left(d t\mb{E}\left[V(X_0)\right]+(\beta-1)^2 t^2 \mb{E}\left[|\nabla V(X_0)|^2\right]\right) dt \nonumber \\
     = & \frac{1}{2} \alpha C_V d (1+r) h\mb{E}[|X_0-x_0|^2]\\ 
     &\qquad +3\alpha C_V d^2 (1+r^{-1}) h^2 \mb{E}\left[ V(X_0) \right]+2\alpha C_V d(\beta-1)^2 (1+r^{-1}) h^3\mb{E}\left[|\nabla V(X_0)|^2\right]. \nonumber
\end{align}
With \eqref{eq:strongly convex V U1} and \eqref{eq:strongly convex V U3}, we get
\begin{align}\label{eq:strongly convex V U1 plus U3}
    &\mb{E}[|U_1+U_3|^2] \nonumber\\
    \le & \left( 1-2(\beta-1)\alpha h+(\beta-1)^2\alpha^2 h^2+\frac{1}{2}\alpha C_V d(1+r) h \right)\mb{E}\left[ |X_0-x_0|^2 \right]  \nonumber \\
    & \quad + 3\alpha C_V d^2 (1+r^{-1}) h^2 \mb{E}\left[ V(X_0) \right] +2\alpha C_V d(\beta-1)^2 (1+r^{-1}) h^3\mb{E}\left[|\nabla V(X_0)|^2\right] \nonumber \\
    \le & \left( 1-2(\beta-1)\alpha h+(\beta-1)^2\alpha^2 h^2+\frac{1}{2}\alpha C_V d(1+r) h \right)\mb{E}\left[ |X_0-x_0|^2 \right]  \nonumber \\
    &\quad +2\alpha C_V d(1+r^{-1})h^2 \left( 3d \mb{E}\left[ V(X_0) \right]+2(\beta-1)^2 h \mb{E}\left[ |\nabla V(X_0)|^2 \right] \right).
\end{align}
Since $C_V<\frac{4(\beta-1)}{d}$, denote $\delta=\frac{(\beta-1)-\frac{1}{4} C_V d}{\frac{1}{4} C_V d}>0$. We have
\begin{align*}
     &1-2(\beta-1)\alpha h+(\beta-1)^2\alpha^2 h^2+\frac{1}{2}\alpha C_V d(1+r) h \\
    =&~1-2(\beta-1)\alpha h+(\beta-1)^2\alpha^2 h^2+2(\beta-1)\alpha\frac{1+r}{1+\delta} h\\
    =&~\left[1-\alpha(\beta-1) (1-\frac{1+2r}{1+\delta})h \right]^2+\alpha^2(\beta-1)^2h^2 \\ &\qquad -2\alpha(\beta-1)\frac{r}{1+\delta}h-\alpha^2(\beta-1)^2 h^2\left(\frac{\delta-2r}{1+\delta}\right)^2.
\end{align*}
By picking $r=\frac{\delta}{3}$, we get for any $h\in \left(0, \frac{2\delta}{3(1+\delta)\alpha(\beta-1)}\right)$ that
\begin{align*}
     &1-2(\beta-1)\alpha h+(\beta-1)^2\alpha^2 h^2+\frac{1}{2}\alpha C_V d(1+r) h \\
     &\quad \le \left[1-\alpha(\beta-1)\frac{\delta}{3(1+\delta)} h \right]^2+\alpha^2(\beta-1)^2h \left( h-\frac{2\delta}{3(1+\delta)}\alpha^{-1}(\beta-1)^{-1} \right) \\
     &\quad \le \left[1-\alpha(\beta-1)\frac{\delta}{3(1+\delta)} h \right]^2.
\end{align*}
With the choice of $r=\delta/3$, \eqref{eq:strongly convex V U1 plus U3} could be rewritten as
\begin{align}\label{eq:strongly convex V U1 plus U3 version 2}
    \mb{E}[|U_1+U_3|^2] &\le \left(1-\frac{\alpha(\beta-1)\delta}{3(1+\delta)}h \right)^2 \mb{E}[|X_0-x_0|^2] \nonumber \\
    &\quad+\frac{8\alpha(\beta-1)(3+\delta)h^2}{(1+\delta)\delta}\left( 3d \mb{E}\left[ V(X_0) \right]+2(\beta-1)^2 h \mb{E}\left[ |\nabla V(X_0)|^2 \right] \right).
\end{align}
Next, with the bound in \eqref{eq:strongly convex V moment bound}, we get when $h<\frac{1}{4(\beta-1)L}$,
\begin{align}\label{eq:strongly convex V U2}
    \mb{E}[|U_2|^2]&\le (\beta-1)^2 L^2 \mb{E}\left[ \left(\int_0^h |X_t-X_0| dt\right)^2 \right] \nonumber\\
    &\le (\beta-1)^2 L^2 h \int_0^h \mb{E}\left[ |X_t-X_0|^2 \right] dt \nonumber\\
    &\le 6d(\beta-1)^2 L^2 h^3 \mb{E}\left[ V(X_0)\right]+4(\beta-1)^4 L^2 h^4 \mb{E}\left[ |\nabla V(X_0)|^2 \right].
\end{align}
With \eqref{eq:strongly convex V U1 plus U3 version 2} and \eqref{eq:strongly convex V U2}, we get when $h<\min\left( \frac{1}{4(\beta-1)L}, \frac{2\delta}{3(1+\delta)\alpha(\beta-1)} \right)$,
\begin{align*}
    \mb{E}\left[|X_h-x_1|^2\right]^{\frac{1}{2}}\le \left[(1-A)^2 \mb{E}\left[|X_0-x_0|^2\right] +B^2\right]^{\frac{1}{2}}+C, 
\end{align*}
with
\begin{align}
    A&=\frac{\alpha(\beta-1)\delta}{3(1+\delta)}h, \label{eq:strongly convex V parameters A}\\
    B&=\frac{4\alpha^{\frac{1}{2}}(\beta-1)^{\frac{1}{2}}(3+\delta)^{\frac{1}{2}}h}{(1+\delta)^{\frac{1}{2}}\delta^{\frac{1}{2}}}\left( d^{\frac{1}{2}} \mb{E}_{\pi_\beta}\left[ V(X) \right]^{\frac{1}{2}}+(\beta-1) h^{\frac{1}{2}} \mb{E}_{\pi_\beta}\left[ |\nabla V(X)|^2 \right]^{\frac{1}{2}} \right), \label{eq:strongly convex V parameters B}\\
    C&=3d^{\frac{1}{2}}(\beta-1) L h^{\frac{3}{2}} \mb{E}_{\pi_\beta}\left[ V(X)\right]^{\frac{1}{2}}+2(\beta-1)^2 L h^2 \mb{E}_{\pi_\beta}\left[ |\nabla V(X)|^2 \right]^{\frac{1}{2}}. \label{eq:strongly convex V parameters C}
\end{align}
The above analysis works for each step, therefore we get for all $k\ge 1$,
\begin{align*}
    \mb{E}\left[|X_{kh}-x_k|^2\right]^{\frac{1}{2}}\le \left[(1-A)^2 \mb{E}\left[|X_{(k-1)h}-x_{k-1}|^2\right] +B^2\right]^{\frac{1}{2}}+C. 
\end{align*}
According to \cite[Lemma 9]{dalalyan2019user}, with $A,B,C$ given in \eqref{eq:strongly convex V parameters A},\eqref{eq:strongly convex V parameters B},\eqref{eq:strongly convex V parameters C}, for all $k\ge 1$,
\begin{align*}
    \mb{E}\left[|X_{kh}-x_k|^2\right]^{\frac{1}{2}}\le (1-A)^k \mb{E}\left[|X_{0}-x_0|^2\right]^{\frac{1}{2}}+\frac{C}{A}+\frac{B}{\sqrt{A(2-A)}}.
\end{align*}
Choosing $X_0$ such that $W_2(\nu_0,\pi_\beta)=\mb{E}\left[|X_{0}-x_0|^2\right]^{\frac{1}{2}}$, we get \eqref{eq:strongly convex V W2 convergence}.\\
\end{proof}

We now prove Theorem \ref{thm:zeroth order algorithm strongly convex V}.\\ 

\begin{proof}[Proof of Theorem \ref{thm:zeroth order algorithm strongly convex V}] Following the same strategy and notation in the proof of Theorem \ref{thm:strongly convex V W2 decay}, we have
\begin{align}\label{eq:difference zeroth order}
    X_h-x_1= U_1+U_2+U_3+(\beta-1)h \mb{E}[\zeta_0|x_0]+(\beta-1)h \left(\zeta_0-\mb{E}[\zeta_0|x_0]\right),
\end{align}
where $U_1,U_2,U_3$ are defined in \eqref{eq:U1},\eqref{eq:U2},\eqref{eq:U3} respectively and $\zeta_0=g_{\sigma,m}(x_0)-\nabla V(x_0)$. Therefore we have
\begin{align}\label{eq:L^2 difference zeroth order}
\begin{aligned}
    \mb{E}\left[ |X_h-x_1|^2 \right]^{\frac{1}{2}}&\le \mb{E}\left[ |U_1+U_3+ (\beta-1)h\left(\zeta_0-\mb{E}[\zeta_0|x_0]\right)|^2 \right]^{\frac{1}{2}} \\
    &\qquad +\mb{E}\left[|U_2|^2\right]^{\frac{1}{2}}+(\beta-1)h\mb{E}\left[ |\mb{E}[\zeta_0|x_0]|^2\right]^{\frac{1}{2}}  \\
    &=\left\{ \mb{E}\left[|U_1+U_3|^2\right]+(\beta-1)^2h^2\mb{E}\left[|\zeta_0-\mb{E}[\zeta_0|x_0]|^2\right] \right\}^{\frac{1}{2}}\\
    &\qquad +\mb{E}\left[|U_2|^2\right]^{\frac{1}{2}}+(\beta-1)h\mb{E}\left[ |\mb{E}[\zeta_0|x_0]|^2\right]^{\frac{1}{2}}.
\end{aligned}
\end{align}
From the proof of Theorem \ref{thm:strongly convex V W2 decay} and Proposition \ref{prop:zeroth order approximator}, when 
$$
h<\min\left(\frac{1}{4(\beta-1)h}, \frac{2\delta}{3(1+\delta)\alpha (\beta-1)} \right),
$$
we have that 
\begin{align}\label{eq:iteration zeroth order}
     \mb{E}\left[ |X_h-x_1|^2 \right]^{\frac{1}{2}}&\le \left\{ (1-A)^2\mb{E}\left[|X_0-x_0|^2\right]+B^2+ \frac{\sigma^2}{2m}L^2(\beta-1)^2(d+3)^3h^2\right. \nonumber \\
     & \left. \quad +\frac{2(d+5)(\beta-1)^2h^2}{m}\mb{E}\left[|\nabla V(x_0)|^2\right] \right\}^{\frac{1}{2}}+C+ L \sigma (\beta-1)d^{\frac{1}{2}} h,
\end{align}
where $A,B,C$ are defined in \eqref{eq:strongly convex V parameters A},\eqref{eq:strongly convex V parameters B},\eqref{eq:strongly convex V parameters C}. Using the fact that $V$ is gradient Lipshcitz, we have
\begin{align}\label{eq:expectation at iterates bound}
    \mb{E}\left[|\nabla V(x_0)|^2\right]&\le \mb{E}\left[ \left(|\nabla V(X_0)|+L|X_0-x_0|\right)^2 \right] \nonumber \\
    &\le 2\mb{E}\left[|\nabla V(X_0)|^2\right]+2L^2\mb{E}\left[|X_0-x_0|^2\right].
\end{align}
Plugging \eqref{eq:expectation at iterates bound} in \eqref{eq:iteration zeroth order}, we get
\begin{align}\label{eq:iteration bound 2 zeroth order}
     \mb{E}\left[ |X_h-x_1|^2 \right]^{\frac{1}{2}}&\le \left\{ (1-A)^2\mb{E}\left[|X_0-x_0|^2\right]+\frac{4(d+5)(\beta-1)^2L^2h^2}{m}\mb{E}\left[|X_0-x_0|^2\right]+B^2\right. \nonumber \\
     & \left. \quad + \frac{\sigma^2}{2m}L^2(\beta-1)^2(d+3)^3h^2 +\frac{4(d+5)(\beta-1)^2h^2}{m}\mb{E}\left[|\nabla V(X_0)|^2\right] \right\}^{\frac{1}{2}} \nonumber\\
    & \quad +C+ L \sigma (\beta-1)d^{\frac{1}{2}} h.
\end{align}
When we pick the step-size such that
$$h<\min\left\{ \frac{2(1+\delta)}{\alpha(\beta-1)\delta},\frac{\alpha m \delta} {24(1+\delta)(\beta-1)(d+5)L^2} \right\},$$
we have
\begin{align*}
    (1-A)^2+\frac{4(d+5)(\beta-1)^2L^2h^2}{m}\le \left(1-\frac{A}{2}\right)^2.
\end{align*}
Therefore we have
\begin{align}\label{eq:iteration bound 3 zeroth order}
    \mb{E}\left[ |X_h-x_1|^2 \right]^{\frac{1}{2}}&\le \left\{ (1-A')^2\mb{E}\left[|X_0-x_0|^2\right]+{B'}^2\right\}^{\frac{1}{2}}+C',
\end{align}
where 
\begin{align}
    A'&=\frac{\alpha(\beta-1)\delta}{6(1+\delta)} h , \label{eq:parameter A'}\\
    B'&=\left(\frac{4\alpha^{\frac{1}{2}}(\beta-1)^{\frac{3}{2}}(3+\delta)^{\frac{1}{2}}h^{\frac{3}{2}}}{(1+\delta)^{\frac{1}{2}}\delta^{\frac{1}{2}}}+\frac{2(\beta-1)(d+5)^{\frac{1}{2}}h}{m^{\frac{1}{2}}} \right) \mb{E}_{\pi_\beta}\left[ |\nabla V(X)|^2 \right]^{\frac{1}{2}}\nonumber \\
    &\quad +\frac{4\alpha^{\frac{1}{2}}(\beta-1)^{\frac{1}{2}}d^{\frac{1}{2}}(3+\delta)^{\frac{1}{2}}h}{(1+\delta)^{\frac{1}{2}}\delta^{\frac{1}{2}}}\ \mb{E}_{\pi_\beta}\left[ V(X) \right]^{\frac{1}{2}} +\frac{\sigma L (\beta-1)(d+3)^{\frac{3}{2}}}{m^{\frac{1}{2}}} h , \label{eq:parameter B'} \\
    C'&=3L (\beta-1)d^{\frac{1}{2}}  h^{\frac{3}{2}} \mb{E}_{\pi_\beta}\left[ V(X)\right]^{\frac{1}{2}}+2L(\beta-1)^2 h^2 \mb{E}_{\pi_\beta}\left[ |\nabla V(X)|^2 \right]^{\frac{1}{2}}+\sigma L (\beta-1)d^{\frac{1}{2}} h.  \label{eq:parameter C'}
\end{align}
The rest of the proof is the same as the proof of Theorem \ref{thm:strongly convex V W2 decay}, and hence we get \eqref{eq:zeroth order W2 decay}.
\end{proof}

\subsection*{Acknowledgements}
YH was supported in part by NSF TRIPODS grant CCF-1934568. TF was supported by the Engineering and Physical Sciences Research Council (EP/T5178) and by the DeepMind scholarship. KB was supported in part by NSF grant DMS-2053918. MAE was supported by NSERC Grant [2019-06167], Connaught New Researcher Award, CIFAR AI Chairs program, and CIFAR AI Catalyst grant. Parts of this work was done when YH, KB and MAE visited the Simons Institute for the Theory of Computing as a part of the ``Geometric Methods in Optimization and Sampling" program during Fall 2021.

\bibliography{reference}
\appendix

\section{Computations for Section~\ref{sec:gencase}}\label{sec:appgencase}

\begin{lemma}\label{lem:ratio between normalization coeffcients} Let $\beta>\frac{d}{2}+1$. If $V\in \mc{C}^2(\mb{R}^d)$ is positive, $\alpha$-strongly convex and $L$-gradient Lipschitz, we have for any $r\in (0,\beta-\frac{d}{2}-1)$,
\begin{align}\label{eq:ratio between normalization coefficients}
    \frac{Z_{\beta-1}}{Z_{\beta}}\le \left(\frac{L}{\alpha}\right)^{\frac{\frac{d}{2}}{\beta-\frac{d}{2}-r}} V(0) \left( \frac{\Gamma(\beta)\Gamma(r)}{\Gamma(\frac{d}{2}+r)\Gamma(\beta-\frac{d}{2})} \right)^{\frac{1}{\beta-\frac{d}{2}-r}}.
\end{align}
\end{lemma}
\begin{proof}\label{pf:ratio between normalization coeffcients} Since $V(x)\le V(0)+\frac{L}{2}|x|^2$, we know that for any $r\in (0,\beta-\frac{d}{2}-1)$, $Z_{\frac{d}{2}+r}$ is finite and $\pi_{\frac{d}{2}+r}$ is a probability measure. Therefore
\begin{align*}
    \frac{Z_{\beta-1}}{Z_{\beta}} &=\frac{\int_{\mb{R}^d} V(x)^{-(\beta-1)} dx }{\int_{\mb{R}^d} V(x)^{-\beta} dx} \\
    &=\frac{ Z_{\frac{d}{2}+r} \int_{\mb{R}^d} V(x)^{-(\beta-\frac{d}{2}-1-r)} \pi_{\frac{d}{2}+r}(x)dx }{ Z_{\frac{d}{2}+r} \int_{\mb{R}^d} V(x)^{-(\beta-\frac{d}{2}-r)} \pi_{\frac{d}{2}+r}(x)dx} \\
    &\le \frac{\left( \int_{\mb{R}^d} V(x)^{-(\beta-\frac{d}{2}-r)} \pi_{\frac{d}{2}+r}(x)dx \right)^{\frac{\beta-\frac{d}{2}-1-r}{\beta-\frac{d}{2}-r}}}{ \int_{\mb{R}^d} V(x)^{-(\beta-\frac{d}{2}-r)} \pi_{\frac{d}{2}+r}(x)dx}\\
    &=\left( \int_{\mb{R}^d} V(x)^{-(\beta-\frac{d}{2}-r)} \pi_{\frac{d}{2}+r}(x)dx \right)^{-\frac{1}{\beta-\frac{d}{2}-r}} \\
    &=\left(Z_{\frac{d}{2}+r}\right)^{\frac{1}{\beta-\frac{d}{2}-r}} \left( \int_{\mb{R}^d} V(x)^{-\beta} dx\right)^{-\frac{1}{\beta-\frac{d}{2}-r}} \\
    &\le \left(Z_{\frac{d}{2}+r}\right)^{\frac{1}{\beta-\frac{d}{2}-r}} \left( \int_{\mb{R}^d} (V(0)+\frac{L}{2}|x|^2)^{-\beta} dx\right)^{-\frac{1}{\beta-\frac{d}{2}-r}}.
\end{align*}
For the integral $\int_{\mb{R}^d} (V(0)+\frac{L}{2}|x|^2)^{-\beta} dx$, we can calculate it via change of polar coordinates and substitutions,
\begin{align*}
    \int_{\mb{R}^d} (V(0)+\frac{L}{2}|x|^2)^{-\beta} dx&= A_{d-1}(1)\int_0^\infty (V(0)+\frac{L}{2}R^2)^{-\beta} R^{d-1} dR \\
    &=\frac{\pi^{\frac{d}{2}}}{\Gamma(\frac{d}{2})} \int_0^\infty (V(0)+V(0) R_L )^{-\beta} (\frac{2V(0)}{L})^{\frac{d}{2}-1} {R_L}^{\frac{d}{2}-1} \frac{2V(0)}{L} dR_L \\
    &=\frac{2^{\frac{d}{2}} \pi^{\frac{d}{2}}}{\Gamma(\frac{d}{2}) L^{\frac{d}{2}}V(0)^{\beta-\frac{d}{2}}} \int_0^\infty (1+R_L)^{-\beta} R_L^{\frac{d}{2}-1}dR_L \\
    &=\frac{2^{\frac{d}{2}} \pi^{\frac{d}{2}}}{\Gamma(\frac{d}{2}) L^{\frac{d}{2}}V(0)^{\beta-\frac{d}{2}}} \int_0^1 u^{\frac{d}{2}-1}(1-u)^{\beta-\frac{d}{2}-1} du \\
    &=\frac{2^{\frac{d}{2}} \pi^{\frac{d}{2}} B(\frac{d}{2},\beta-\frac{d}{2}) }{\Gamma(\frac{d}{2}) L^{\frac{d}{2}}V(0)^{\beta-\frac{d}{2}}},
\end{align*}
where the second identity follows from a substitution with $R_L=LR^2/(2V(0))$ and the fourth identity follows from a substitution with $u=\frac{R_L}{1+R_L}$.
For $Z_{\frac{d}{2}+r}$, we have
\begin{align*}
    Z_{\frac{d}{2}+r}&=\int_{\mb{R}^d} V(x)^{-\frac{d}{2}-r} dx \\
    &\le \int_{\mb{R}^d } \left( V(0)+\frac{\alpha}{2} |x|^2 \right)^{-\frac{d}{2}-r} dx \\
    &= \frac{\pi^{\frac{d}{2}}}{\Gamma(\frac{d}{2})} \int_0^\infty \left( V(0)+\frac{\alpha}{2} R^2 \right)^{-\frac{d}{2}-r} R^{d-1} dR \\
    &=\frac{\pi^{\frac{d}{2}}}{\Gamma(\frac{d}{2})} \int_0^\infty \left( V(0)+V(0)R_\alpha \right)^{-\frac{d}{2}-r} \left(\frac{2V(0)}{\alpha}\right)^{\frac{d}{2}-1} {R_{\alpha}}^{\frac{d}{2}-1} \frac{2V(0)}{\alpha} d R_\alpha \\
     &=\frac{2^{\frac{d}{2}}\pi^{\frac{d}{2}}}{\Gamma(\frac{d}{2})\alpha^{\frac{d}{2}}V(0)^r} \int_0^\infty \left( 1+R_\alpha \right)^{-\frac{d}{2}-r} {R_{\alpha}}^{\frac{d}{2}-1}  d R_\alpha\\
    &=\frac{2^{\frac{d}{2}} \pi^{\frac{d}{2}}}{\Gamma(\frac{d}{2}) \alpha^{\frac{d}{2}}} \int_0^1 u^{\frac{d}{2}-1}(1-u)^{r-1} du \\
    &=\frac{2^{\frac{d}{2}} \pi^{\frac{d}{2}} B(\frac{d}{2},r) }{\Gamma(\frac{d}{2}) \alpha^{\frac{d}{2}}V(0)^r}.
\end{align*}
Therefore, we can further get
\begin{align*}
    \frac{Z_{\beta-1}}{Z_\beta}&\le \left( \frac{2^{\frac{d}{2}} \pi^{\frac{d}{2}} B(\frac{d}{2},r) }{\Gamma(\frac{d}{2}) \alpha^{\frac{d}{2}}V(0)^r}   \frac{\Gamma(\frac{d}{2}) L^{\frac{d}{2}}V(0)^{\beta-\frac{d}{2}}}{2^{\frac{d}{2}} \pi^{\frac{d}{2}} B(\frac{d}{2},\beta-\frac{d}{2}) } \right) ^{\frac{1}{\beta-\frac{d}{2}-r}}\\
    &=\left( \frac{L^{\frac{d}{2}}V(0)^{\beta-\frac{d}{2}-r} }{\alpha^{\frac{d}{2}}} \frac{\Gamma(\beta)\Gamma(r)}{\Gamma(\frac{d}{2}+r)\Gamma(\beta-\frac{d}{2})} \right)^{\frac{1}{\beta-\frac{d}{2}-r}} \\
    &=\left(\frac{L}{\alpha}\right)^{\frac{\frac{d}{2}}{\beta-\frac{d}{2}-r}} V(0) \left( \frac{\Gamma(\beta)\Gamma(r)}{\Gamma(\frac{d}{2}+r)\Gamma(\beta-\frac{d}{2})} \right)^{\frac{1}{\beta-\frac{d}{2}-r}}.
\end{align*}
\end{proof}

\section{Computations for Sections \ref{subsec:cauchy large dof} and \ref{subsec:cauchy small dof}}\label{cauchyexample}

Let $\pi_\beta(x)\propto V(x)^{-\beta}= (1+|x|^2)^{-\beta}$ with $\beta>\frac{d+2}{2}$. The gradient and Hessian of $V$ is 
\begin{align*}
    \nabla V(x)=2x, \qquad \nabla^2 V(x)=2 I_d.
\end{align*}
Therefore $V$ is $\alpha$-strongly convex with $\alpha=2$ and $L$-gradient Lipschitz with $L=2$. \eqref{eq:diffusion strongly convex V} reduces to \begin{equation}\label{eq:diffusion cauchy}
    d X_t= b(x) dt + \sigma(X_t) dB_t,
\end{equation} 
with $b(x)=-2(\beta-1)x$ and $\sigma(x)=\sqrt{2}(1+|x|^2)^{\frac{1}{2}}I_d$.\\

Next we look at the uniform dissipativity condition:
\begin{align}\label{eq:example cauchy uniform dissipativity}
      &\langle b(x)-b(y),x-y \rangle+\frac{1}{2}\lv (1+|x|^2)^{\frac{1}{2}}I_d-(1+|y|^2)^{\frac{1}{2}}I_d \rv_F^2 \nonumber \\
        =&-2(\beta-1)|x-y|^2+d|(1+|x|^2)^{\frac{1}{2}}-(1+|y|^2)^{\frac{1}{2}}|^2  \nonumber \\
        \le & -2(\beta-1-\frac{d}{2}) |x-y|^2,
\end{align}
where the inequality follows from the fact that $x\mapsto (1+|x|^2)^{\frac{1}{2}}$ is $1$-Lipschitz. Therefore diffusion \eqref{eq:diffusion cauchy} is $\alpha'$-uniform dissipative with $\alpha'=2(\beta-1-\frac{d}{2})$. In particular, $\alpha'=d$ when $\beta=d+1$ and $\alpha'=1$ when $\beta=\frac{d+3}{2}$.\\

Last we look at the local deviation for the Euler discretization to \eqref{eq:diffusion cauchy}. We use the same notations in \cite{li2019stochastic}. According to \cite[lemma 29]{li2019stochastic}, $p_1=1$ and 
\begin{align*}
    \lambda_1=2\left( \mu_1(b)^2+\mu_1^F(\sigma)^2 \right)\left( \pi_{1,2}(b)+\pi_{1,2}^F(\sigma) \right)(1+\mb{E}[|\Tilde{X}_0|^2]+2\pi_{1,2}(b){\alpha'}^{-1}).
\end{align*}
According to \cite[lemma 29]{li2019stochastic}, $p_2=\frac{3}{2}$ and 
\begin{align*}
    \lambda_2=\mu_1(b)\left( \pi_{1,2}(b)+\pi_{1,2}^F(\sigma) \right)(1+\mb{E}[|\Tilde{X}_0|^2]+2\pi_{1,2}(b){\alpha'}^{-1}),
\end{align*}
with
\begin{align*}
    \mu_1(b)&:=\sup_{x,y\in \mb{R}^d,x\neq y} \frac{|b(x)-b(y)|}{|x-y|}=2(\beta-1),\\
    \mu_1^F(\sigma)&:=\sup_{x,y\in \mb{R}^d,x\neq y} \frac{\lv \sigma(x)-\sigma(y)\rv_F}{|x-y|}=\sqrt{2d},\\
    \pi_{1,2}(b)&:=\sup_{x\in \mb{R}^d} \frac{|b(x)|^2}{1+|x|^2}=4(\beta-1)^2, \\
    \pi_{1,2}^F(\sigma)&:=\sup_{x\in \mb{R}^d} \frac{\lv \sigma(x) \rv_F^2}{1+|x|^2}=2d.
\end{align*}
The order of $\lambda_1$ and $\lambda_2$ in dimension parameter $d$ is given by:
\begin{align*}
    \lambda_1&= \Theta\left( \left((\beta-1)^2+d\right)\left((\beta-1)^2+2d\right) \left(1+(\beta-1)^2{\alpha'}^{-1}\right) \right),\\
    \lambda_2&=\Theta\left( \left(\beta-1\right)\left((\beta-1)^2+2d\right) \left(1+(\beta-1)^2{\alpha'}^{-1}\right) \right).
\end{align*}
Therefore, we have that
\begin{itemize}
    \item when $\beta=d+1$, $(\lambda_1,\lambda_2)=(\Theta(d^5),\Theta(d^4))$,
    \item when $\beta=\frac{d+3}{2}$, $(\lambda_1,\lambda_2)=(\Theta(d^5),\Theta(d^4))$.
\end{itemize}

\section{Computations for Remark~\ref{rem:cauchywpi}}\label{sec:cauchy application} In the example of Cauchy class distributions, $V(x)=1+|x|^2$ and $V_{\gamma}:=V^{\gamma}$. When $\gamma>\frac{1}{2}$,
\begin{align*}
    \nabla V_{\gamma}(x)&=\gamma V(x)^{\gamma-1} \nabla V(x) \\
    \nabla^2 V_\gamma(x)&= \gamma(\gamma-1)V(x)^{\gamma-2} \nabla V(x)^T \nabla V(x)+\gamma V(x)^{\gamma-1}\nabla^2 V(x) \\
    &=\gamma V(x)^{\gamma-1}\left( (\gamma-1)V(x)^{-1}\nabla V(x)^T\nabla V(x)+\nabla^2 V(x) \right).
\end{align*}
Plug in $V(x)=1+|x|^2$, we get
\begin{align*}
    \nabla V_\gamma (x)&=2\gamma (1+|x|^2)^{\gamma-1} x \\
    \nabla^2 V_\gamma(x)&=2\gamma (1+|x|^2)^{\gamma-1} \left( I_d+2(\gamma-1) \frac{|x|^2}{1+|x|^2} \frac{x^Tx}{|x|^2} \right)\\
    &=2\gamma (1+|x|^2)^{\gamma-1} \left( (I_d-\frac{x^Tx}{|x|^2})+\left(1-2(1-\gamma) \frac{|x|^2}{1+|x|^2}\right) \frac{x^Tx}{|x|^2} \right),
\end{align*}
and 
\begin{align*}
    (\nabla^2 V_\gamma)^{-1}(x)=\frac{1}{2\gamma} (1+|x|^2)^{1-\gamma} \left( (I_d-\frac{x^Tx}{|x|^2})+ \frac{1+|x|^2}{1+(2\gamma-1)|x|^2} \frac{x^Tx}{|x|^2}  \right).
\end{align*}
When $\beta\in \left(\frac{d+2}{2},d\right]$, $\gamma=\frac{\beta}{d+2}\in \left(\frac{1}{2},1\right]$,
\begin{align*}
    (\nabla^2  V_\gamma)^{-1}(x) \preceq \frac{1}{2\gamma(2\gamma-1)} (1+|x|^2)^{1-\gamma} I_d =\frac{(d+2)^2}{2\beta(2\beta-d-2)}  (1+|x|^2)^{1-\gamma} I_d.
\end{align*}
Therefore Assumption \ref{ass:WPI on V small beta} holds with $C_V(\gamma)=\frac{(d+2)^2}{2\beta(2\beta-d-2)}$. For the Cauchy distribution $\pi_\beta\propto (1+|x|^2)^{-\beta}=(1+|x|^2)^{-\frac{d+\nu}{2}}$ with $\beta\in (\frac{d+2}{2},d]$, i.e. $\nu\in(2,d]$, according to lemma \ref{lem:WPI small beta}, $\pi_\beta$ satisfies the weighted Poincar\'{e} inequality with weight $1+|x|^2$ with weighted Poincar\'{e} constant
\begin{align*}
    \cwpi= C_V(\gamma)\left(\frac{\beta}{\gamma}-1\right)^{-1}=\frac{(d+2)^2}{2(d+1)\beta(2\beta-d-2)}=\frac{(d+2)^2}{\nu(d+1)(d+\nu)}.
\end{align*}

\end{document}